\documentclass[12pt]{amsart}
\usepackage[margin=3.5cm]{geometry}
\usepackage{amscd}
\usepackage{amsfonts,amssymb,amsthm}
\usepackage{graphics}
\usepackage{url}

\newcommand{\Isom}{\mathrm{Isom}\, }

\newcommand{\Fix}{\mathrm{Fix}\, }

\usepackage{amsmath}

\usepackage{amsthm}
\usepackage[dvips]{graphicx}

\usepackage{amscd}
\usepackage{epsfig}
\usepackage{psfrag}
\usepackage{srcltx}

\usepackage{amssymb}
\usepackage{latexsym}

\newtheorem{Theorem}{Theorem}[section]

\newtheorem{Proposition}[Theorem]{Proposition}
\newtheorem{Lemma}[Theorem]{Lemma}

\newtheorem{Remark}[Theorem]{Remark}
\newtheorem{Claim}[Theorem]{Claim}

\theoremstyle{definition}
\newtheorem{definition}[Theorem]{Definition}
\newtheorem{remark}[Theorem]{Remark}

\def\OO{{\mathcal O}}

\def\G{{\mathbb G}}
\def\C{{\mathbb C}}
\def\R{{\mathbb R}}

\def\Z{{\mathbb Z}}
\def\Q{{\mathbb Q}}

\def\OO{{\mathcal O}}

\def\G{{\mathbb G}}
\def\C{{\mathbb C}}
\def\R{{\mathbb R}}

\def\Z{{\mathbb Z}}
\def\Q{{\mathbb Q}}

\def\R{{\mathbb R}}

\def\threesphere{{{\mathbf S}^3}}
\def\twosphere{{{\mathbf S}^2}}

\def\zerosphere{{{\mathbf S}^0}}
\def\branchedcover{{{\tilde{\mathbf S}}^3}}
\def\threeball{{{\mathbf B}^3}}

\def\interior{{\mathrm{int}}}

\def\Fix{{\mathrm{Fix}}}
\def\Diff{{\mathrm{Diff}}}

\begin{document}

\title{Prime amphicheiral knots with free period $2$}

\author{Luisa Paoluzzi}
\address{Aix-Marseille Univ, CNRS, Centrale Marseille, I2M, UMR 7373,
13453 Marseille, France}
\email{luisa.paoluzzi@univ-amu.fr}

\author{Makoto Sakuma}
\address{Department of Mathematics\\
Graduate School of Science\\
Hiroshima University\\
Higashi-Hiroshima, 739-8526, Japan}
\email{sakuma@hiroshima-u.ac.jp}

\date{\today}

\subjclass[2010]{Primary 57M25; Secondary 57M50 \\
\indent {
The first author was supported by JSPS Invitational Fellowship for Research in 
Japan S17112.
The second author was supported by JSPS Grants-in-Aid 15H03620.}}

\maketitle

\begin{abstract}
\vskip 2mm

We construct prime amphicheiral knots that have free period $2$.
This settles an open question raised by the second named author, who proved 
that amphicheiral hyperbolic knots cannot admit free periods and that prime 
amphicheiral knots cannot admit free periods of order $>2$.
\vskip 2mm


\vskip 2mm


\end{abstract}

\section{Introduction}
\label{section:intro}
A knot $K$ in the $3$-sphere is \emph{amphicheiral} if there exists 
an orientation-reversing diffeomorphism $\varphi$ of the $3$-sphere which leaves the knot
invariant. More precisely the knot is said to be \emph{positive-amphicheiral}
(\emph{$+$-amphicheiral} for short) if $\varphi$ preserves a fixed orientation
of $K$ and \emph{negative-amphicheiral} (\emph{$-$-amphicheiral} for short)
otherwise. Of course, a knot can be both positive- and negative-amphicheiral.
This happens if and only if the knot is both amphicheiral and invertible. 
Recall that a knot is \emph{invertible} if there exists 
an orientation-preserving diffeomorphism of the $3$-sphere which leaves the knot 
invariant but reverses its orientation.

A knot $K$ in the $3$-sphere is said to have
\emph{free period $n$} $\ge 2$ if there exists an
orientation-preserving periodic diffeomorphism $f$ of the $3$-sphere of order $n$
which leaves the knot invariant, such that $f$ generates a free $\Z/n\Z$-action on the $3$-sphere.

It was proved by the second named author that amphicheirality and free periodicity are inconsistent 
in the following sense (see \cite{S1, S2}).

\begin{enumerate}
\item 
Any amphicheiral prime knot does not have free period $>2$.
\item
Any amphicheiral hyperbolic knot does not have free period.
\end{enumerate}
On the other hand, it is easy to construct, for any integer $n\ge 2$, a 
composite amphicheiral knot that has free period $n$. 
Thus it was raised in \cite{S2} as an open question whether there is an 
amphicheiral prime knot that has free period $2$.

The main purpose of this paper is to prove the following theorem
which gives an answer to this question.

\begin{Theorem}
\label{thm1}
{\rm (1)} For each $\epsilon\in\{+,-\}$, there are infinitely many prime knots
with free period $2$ that are $\epsilon$-amphicheiral but not
$-\epsilon$-amphicheiral; in particular, they are not invertible.

{\rm (2)} There are infinitely many prime knots with free period $2$ that are
$\epsilon$-amphicheiral for each $\epsilon\in\{+,-\}$; in particular, they are
invertible.
\end{Theorem}

For the proof of the theorem,
we introduce a specific subgroup 
$
\G=\langle \gamma_1, \gamma_2\rangle
$
of the orthogonal group $\mathrm{O}(4) \cong\mathrm{Isom}(\threesphere)$,
generated by two commuting orientation-reversing involutions $\gamma_1$ and 
$\gamma_2$, where $\gamma_1$ is a reflection in a $2$-sphere, 
$\gamma_2$ is a reflection in a $0$-sphere, 
and $f:=\gamma_1\gamma_2$ is a free involution.
Consider a $\G$-invariant hyperbolic link 
$L=K_0\cup \OO_{\mu}$ in $\threesphere$, consisting of a specific 
$\G$-invariant component $K_0$ and a $\mu$-component trivial link $\OO_{\mu}$. 
Then we prove that given such a link $L$,
we can construct a prime amphicheiral knot $K$ with free period $2$,
such that the exterior $E(L)$ of $L$ is the \emph{root} $E_0$ of the JSJ 
decomposition of the exterior $E(K)$, 
i.e. the geometric piece containing the boundary of $E(K)$,
where the boundary torus of the tubular neighbourhood of $K_0$
corresponds to $\partial E(K)$.
In fact, we show that each 
$\gamma_i$ 
(to be precise, the restriction of $\gamma_i$ to $E(L)$)
extends to an orientation-reversing diffeomorphism of $\threesphere$ preserving 
the prime knot $K$, and 
$f$ extends to a free involution on $\threesphere$ preserving $K$.
It should be noted that $\G$ does not necessarily 
extend to a group action on $(\threesphere, K)$.
In fact, the extension of $\gamma_1$ generically has infinite order in the symmetry group $\pi_0\Diff(\threesphere,K)$
(see Remark~\ref{rem:non-periodic}).

The proof of Theorem~\ref{thm1} is then reduced to producing examples of
links $L$ fulfilling the above properties. Of course, while it is not
difficult to construct links admitting prescribed symmetries, ensuring that
they are hyperbolic can be much more delicate. We will construct three
links that will provide examples of knots having different properties. We will
give theoretic proofs of the fact that they are hyperbolic, although this can
also be checked using the computer program SnapPea \cite{Weeks}, SnapPy \cite{CDG}, or 
the computer verified program 
HIKMOT \cite{HIKMOT}.

Moreover, we show that any amphicheiral prime knot with free period $2$ is constructed in this way
(Theorem~\ref{prop:structure}).
In other words, if $K$ is an amphicheiral prime knot with free period $2$,
then the root $E_0$ of the JSJ decomposition of $E(K)$ is identified with $E(L)$
for some $\G$-invariant link $L$ with the above property.
Furthermore, we prove the following theorem 
which provides some insight on the root $E_0$
with respect to the $\G$-action.

\begin{Theorem}
\label{thm2}
Let $K$ be a prime amphicheiral knot with free period $2$ and let
$E_0=E(L)=E(K_0\cup\OO_{\mu})$ be its root. Then after an isotopy, $L$ is 
invariant by the action of $\G$ on $\threesphere$, and the following hold.

{\rm (1)} $L$ contains

\begin{itemize}
\item at most two components whose stabiliser is $\G$, one of which must be
$K_0$;
\item at least one pair of components of $\OO_{\mu}$, 
such that each of the components intersects the $2$-sphere $\Fix(\gamma_1)$ 
transversely in two points, 
the stabiliser of each of the components is generated by $\gamma_1$,
and 
$f$ interchanges the two components;
\item no component with stabiliser generated by $f$.
\end{itemize}

{\rm (2)} Assume that $K$ is positive-amphicheiral. Then $K_0$ is contained in
$\Fix(\gamma_1)$.

{\rm (3)} Assume that $K$ is negative-amphicheiral. Then $K_0$ must contain
$\Fix(\gamma_2)$ and intersect transversally $\Fix(\gamma_1)$ in two points.
\end{Theorem}

\begin{Remark}
\label{rem:invertible}
{\rm
In the above theorem, if $K$ is both positive- and negative-amphi\-chei\-ral,
then $L$ admits two different 
positions with respect to the action of $\G$.
In other words,
there are two subgroups $\G_+$ and $\G_-$ in $\Diff(S^3,L)$
such that (i) both $\G_+$ and $\G_-$ are conjugate to $\G$ in $\Diff(S^3)$ and 
(ii) the groups $\G_+$ and $\G_-$ satisfy the conditions (2) and (3), respectively.
In fact, $(S^3,L)$ admits an action of $(\Z/2\Z)^3$ such that
$\G_+$ and $\G_-$ 
correspond to $(\Z/2\Z)^2\oplus 0$ and 
$0\oplus (\Z/2\Z)^2$, respectively,
where $0\oplus (\Z/2\Z)\oplus 0$ is generated by $f$.
}
\end{Remark}

For a prime knot $K$, let $\mu(K)$ be the number of boundary components of
the root $E_0$ of the JSJ decomposition of $E(K)$
that correspond to the tori of the JSJ decomposition,
i.e., $\mu(K)+1$ is equal to the number of boundary components of $E_0$. 
Define now $\mu_{+}$ (respectively $\mu_{-}$) to be the minimum of $\mu(K)$
over all prime knots $K$ with free period $2$ that are positive- (respectively
negative-) amphicheiral. 
The main result of \cite{S2} says that $\mu_{\pm}>0$.
The following theorem determines both $\mu_+$ and $\mu_-$.

\begin{Theorem}
\label{thm:minimal}
The following hold:

{\rm (1)} $\mu_{-}=2$.
Moreover, if $K$ realises $\mu_{-}$, i.e., 
if $K$ is a prime negative-amphicheiral knot with free period $2$  
such that $E_0=E(K_0\cup\OO_2)$, then $K_0$
contains $\Fix(\gamma_2)$ and the stabiliser of each component of $\OO_2$ is
generated by $\gamma_1$.

{\rm (2)} $\mu_{+}=3$.
Moreover, if $K$ realises $\mu_{+}$, i.e.,
if $K$ is a prime positive-amphicheiral knot with free period $2$  
such that $E_0=E(K_0\cup\OO_3)$, then $K_0$
is a (necessarily trivial) knot contained in $\Fix(\gamma_1)$, one of the three
components of $\OO_3$ is stabilised by $\G$ and thus contains $\Fix(\gamma_2)$,
while the stabiliser of the other two components is generated by $\gamma_1$.
\end{Theorem}

We remark that our explicit constructions provide in particular negative-
(respectively positive-) amphicheiral knots $K$ with free period $2$ such that
$\mu(K)=2$ (respectively $\mu(K)=3$).
We also describe the structure of the subgroup $\Isom^*(E_0)$ of $\Isom(E_0)$
consisting of those elements which extend to a diffeomorphism of $(\threesphere,K)$
(Proposition \ref{c:isom}).

\bigskip

The paper is organised as follows. In Section~\ref{section:knot}, we show how
one can construct prime amphicheiral knots with free period $2$ from the
exterior of a hyperbolic link with specified properties. In
Section~\ref{section:roots}, we provide examples of such links. 
In Section~\ref{section:uniqueness}, we show that the requirement on the links 
are not only sufficient but also necessary 
(Theorem \ref{prop:structure}).
This is used in 
Section~\ref{section:thm2} to prove Theorems~\ref{thm2} and \ref{thm:minimal}; 
Theorem~\ref{thm1} is also proved in this section. 
In Section~\ref{section:additional-information},
we refine the arguments in Section~\ref{section:thm2}
and give more detailed information concerning the root $E_0$.
In particular, we present a convenient description of the link $L$
in case $K$ is positive amphicheiral 
(Remark~\ref{rem:arc-presentation}).
The final 
Sections~\ref{section:link2}, \ref{section:link3}, and \ref{section:link6} are 
technical: there we show that the links introduced in
Section~\ref{section:roots} are hyperbolic, completing the proof that they
fulfill all the desired requirements. 


\section{Constructing the knots}
\label{section:knot}


In this section, we show that the existence of a hyperbolic link with a
specified symmetry and some extra properties is sufficient to ensure the
existence of prime amphicheiral knots with free period $2$. We start by
defining the symmetry we want.

Let $\G=\langle \gamma_1,\gamma_2\rangle$ be the group of isometries of
the $3$-sphere
$\threesphere=\{(z_1,z_2)\in \C^2 \ | \ |z_1|^2+|z_2|^2 =1\}$
with the standard metric,
generated by the following two commuting orientation-reversing isometric involutions
$\gamma_1$ and $\gamma_2$.
\begin{equation}
\label{group_action}
\gamma_1(z_1,z_2)=(\bar z_1,z_2), \quad
\gamma_2(z_1,z_2)=(-\bar z_1,-z_2).
\end{equation}
The involution $\gamma_1$ is a reflection in the $2$-sphere
$\threesphere\cap (\R\times\C)$,
and the involution $\gamma_2$ is a reflection
in the $0$-sphere $\{(\pm i,0)\}$.
The composition $f:=\gamma_1\gamma_2$
is an orientation-preserving free involution
$(z_1,z_2)\mapsto (-z_1,-z_2)$.
If we identify $\threesphere$ with $\R^3\cup\{\infty\}$, where $\infty$
corresponds to the point $(i,0)\in \threesphere$,
then $\gamma_1$ is an inversion in a unit $2$-sphere in $\R^3$
and $\gamma_2$ is an antipodal map with respect to the origin.
If we identify $\threesphere$ with $\R^3\cup\{\infty\}$, where $\infty$
corresponds to the point $(1,0)\in \threesphere$,
then $\gamma_1$ is a reflection in a $2$-plane in $\R^3$
and $\gamma_2$ has a unique fixed point 
in each of the two half-spaces defined by the $2$-plane. 

\begin{definition}
\label{d:root}
Let $L=K_0\cup \OO_{\mu}$ be a link. We will say that \emph{$L$ provides an
admissible root} if it satisfies the following requirements:
\begin{itemize}
\item $L$ is $\G$-invariant and $K_0$ is 
stabilised by the whole group $\G$.
\item $\OO_{\mu}$ is a trivial link with $\mu$ components.
\item $L$ is hyperbolic.
\end{itemize}
\end{definition}

A link $L$ providing an admissible root can be used to construct a prime
satellite knot in the following way: remove a $\G$-invariant open regular neighbourhood of
$\OO_{\mu}=\cup_{i=1}^\mu O_i$ and glue non-trivial knot exteriors, $E(K_i)$,
$i=1,\dots,\mu$, along the corresponding boundary components in such a way that
the longitude and meridian of $O_i$ are identified with the meridian and
longitude of $E(K_i)$ respectively. The image of $K_0$ in the resulting
manifold is a prime knot $K$ in $\threesphere$.

Note that
\[
(\threesphere,K)=(E(\OO_{\mu}), K_0)\cup ( \cup_{i=1}^{\mu} (E(K_i), \emptyset)),
\]
where $E(\OO_{\mu})$ is the exterior of $\OO_{\mu}$ which is $\G$-invariant.

We wish to perform the gluing so that the following conditions hold.
\begin{enumerate}
\item[(i)]
Each element of $\G$ \lq\lq extends'' to a
diffeomorphism of $(\threesphere, K)$.
To be precise, for each element $g\in\G$,
the restriction, $\check g$, of $g$  
to $(E(\OO_{\mu}),K_0)$ extends to a
diffeomorphism, $\hat g$, of $(\threesphere, K)$.
\item[(ii)]
An \lq\lq extension'' $\hat f$ of $f=\gamma_1\gamma_2$ to $(\threesphere, K)$ is a free involution.
\end{enumerate}
Note that we do not intend to extend the action of $\G$ to $(\threesphere,K)$.
In fact, such an extension does not exist generically
(see Remarks~\ref{rem:non-periodic} and \ref{r:strongly-invertible}, and
Proposition~\ref{p:strongly-invertible}).

To this end we need to choose the knot exterior $E(K_i)$ appropriately
according to the stabiliser in $\G$ of the component $O_i$ of the link $L$.

\medskip
{\bf Case 1.} \emph{The stabiliser of $O_i$ is trivial.} In this case, one
can choose $K_i$ to be any non-trivial knot in $\threesphere$. The same
knot exterior must be chosen for the other three components of $\OO_{\mu}$ in the
same $\G$-orbit as $O_i$ and the gluing must be carried out in a $\G$-equivariant way.
Note that this is possible because the elements of $\G$ map meridians
(respectively longitudes) of the components of $L$ to meridians (respectively
longitudes).

\medskip
{\bf Case 2.} \emph{The stabiliser of $O_i$ is generated by $\gamma_j$ for some $j\in\{1,2\}$.} 
Let $T_i$ be the boundary component of $E(\OO_{\mu})$ which forms the
boundary of a regular neighbourhood of $O_i$, and let $\ell_i$ and $m_i$ be the
longitude and the meridian 
curves in $T_i=\partial E(K_i)$ of the knot $K_i$.
Recall that these are the meridian and longitude of $O_i$ respectively.
We see that $\gamma_j$ acts on $H_1(T_i;\Z)$ either as
\[
(\gamma_j)_*\begin{pmatrix}\ell_i \\ m_i\end{pmatrix}=
\begin{pmatrix}\ell_i \\ -m_i\end{pmatrix}
\]
or
\[
(\gamma_j)_*\begin{pmatrix}\ell_i \\ m_i\end{pmatrix}=
\begin{pmatrix}-\ell_i \\ m_i\end{pmatrix}.
\]
In the former case, we need to choose $K_i$ to be a positive-amphicheiral 
knot in $\threesphere$, while in the latter case we need to
choose $E(K_i)$ to be the exterior of a negative-amphicheiral knot. With this
choice, the restriction of the involution $\gamma_j$ to $(E(\OO_{\mu}), K_0)$
extends to a diffeomorphism of $(E(\OO_{\mu})\cup E(K_i), K_0)$.

This can be seen as follows.
Because of the chosen type of chirality of $K_i$,
there is an orientation-reversing diffeomorphism, $\nu_i$, of $E(K_i)$
whose induced action on $H_1(T_i;\Z)$ is equal to that of $(\gamma_j)_*$.
Then the restrictions of $\gamma_j$ and $\nu_i$ to $T_i$ are smoothly
isotopic. So, by using a collar neighbourhood of $T_i$,
we can glue the diffemorphisms $\gamma_j$ and $\nu_i$
to obtain the desired diffeomorphism of $(E(\OO_{\mu})\cup E(K_i), K_0)$.

For the other component $O_{i'}=f(O_i)$
of $L$ in the $\G$-orbit of $O_i$,
we glue a copy of the same knot exterior along $O_{i'}$,
so that the free involution extends to a free involution
of $(E(\OO_{\mu})\cup E(K_i)\cup E(K_{i'}), K_0)$,
which exchanges $E(K_i)$ with $E(K_{i'})$.
By the previous argument, $\gamma_j$ also extends to a
diffeomorphism of $(E(\OO_{\mu})\cup E(K_i)\cup E(K_{i'}), K_0)$,
and so do all elements of $\G$.

\medskip
{\bf Case 3.} \emph{The stabiliser of $O_i$ is generated by $f$.} This case
cannot arise as shown by the following lemma.

\begin{Lemma}
\label{lem:no-f-stab}
Let $L$ be a link which is invariant by the action of $\G$. Then no component of
$L$ can be stabilised precisely by the cyclic subgroup of $\G$ generated by $f$.
\end{Lemma}

\begin{proof}
Assume that $L_j$ is a component whose stabiliser contains $f$. Since the two
balls of $\threesphere\setminus \Fix(\gamma_1)$ are exchanged by $f$, $L_j$
cannot be contained in one of them, and must thus intersect $\Fix(\gamma_1)$. It
follows that $\gamma_1$ belongs to the stabiliser of $L_j$,
and hence the stabiliser of $L_j$ is the whole group $\G$.
\end{proof}

Note that for the conclusion of the lemma to be valid, we do not need to assume
that $L$ is hyperbolic.

\medskip
{\bf Case 4.} \emph{The stabiliser of $O_i$ is $\G$.}
In this case we choose $K_i$ to be a non-trivial amphicheiral knot with free period $2$; 
recall that, without appealing to Theorem~\ref{thm1}, there
are composite knots with these properties. 
In fact, the connected sum of two copies of
an $\epsilon$-amphicheiral knot is $\epsilon$-amphicheiral and has free period $2$
(cf. \cite[Theorem 4]{S1}).
Now note that the action of $f$ on $H_1(T_i)$ is trivial and therefore
the actions of $\gamma_1$ and $\gamma_2$ on $H_1(T_i)$ are identical.
Hence we may choose $E(K_i)$ to be the exterior of a positive- or negative-amphicheiral knot accordingly,
so that $E(K_i)$ admits an orientation-reversing diffeomorphism, $\nu_i$,
such that the action of $\nu_i$ on $H_1(T_i)$ is identical with
those of $\gamma_1$ and $\gamma_2$.
Thus both $\gamma_1$ and $\gamma_2$ 
extend to diffeomorphisms of
$(E(\OO_{\mu})\cup E(K_i), K_0)$, as in Case 2.

Though the composition of the extensions of $\gamma_1$ and $\gamma_2$ may not be an involution,
we can show that $f=\gamma_1\gamma_2$ extends to a free involution
on $(E(\OO_{\mu})\cup E(K_i), K_0)$.
To see this, we use the fact that the strong equivalence class 
of a free $\Z/n\Z$-action on a torus $T$ is determined by its {\it slope},
which is the submodule of $H_1(T;\Z/n\Z)$ isomorphic to $\Z/n\Z$,
represented by a simple loop on $T$ which is obtained as an orbit of a free circle action
in which the $\Z/n\Z$-action embeds (see \cite[Section 2]{S1}).
Here two actions are {\it strongly equivalent} if they are conjugate by a diffeomorphism
which is isotopic to the identity.
Now, as in Case 1, let $\ell_i$ and $m_i$ be the longitude and the meridian
in $T_i=\partial E(K_i)$ of the knot $K_i$.
Then since $f$ gives a free period $2$ of the trivial knot $O_i$,
the slope of the action of $f$ on $T_i$ is generated by $\ell_i+m_i$
by \cite[Lemma 1.2(3)]{S1}.
Let $f_i$ be a free involution on $E(K_i)$ which realise the free periodicity of $K_i$.
Then the slope of $f_i$ on $T_i$ is also generated by $\ell_i+m_i$.
Thus the restrictions of $f$ and $f_i$ to $T_i$ are strongly equivalent
by \cite[Lemma 1.1]{S1}.
Hence,  by using a collar neighbourhood of $T_i$, 
we can glue these free involutions to obtain a free involution on $(E(\OO_{\mu})\cup E(K_i), K_0)$.

\medskip
The above case-by-case discussion shows that,
given a link $L=K_0\cup\OO_{\mu}$ which provides an admissible root,
we can construct a knot $(\threesphere,K)$
by attaching suitable knot exteriors $\{ E(K_i)\}_{1\le i\le\mu}$
to $(E(\OO_{\mu}),K_0)$,
such that all elements of $\G$ extend to diffeomorphisms of
$(\threesphere,K)$
where an extension of $f$ is a free involution.
It is now clear that the resulting knot $K$ is a prime amphicheiral knot and
has free period $2$.

\begin{remark}
\label{rem:non-periodic}
According to Theorem~\ref{thm2} (see also Claim~\ref{c:atleast2}), $L$ must
contain a component $O_i$ whose stabiliser is
generated by $\gamma_1$, where $\gamma_1$ acts on the boundary torus $T_i$
as a reflection in two meridians of $O_i$.
Thus the knot $K_i$ must be positive amphicheiral.
The orientation-reversing diffeomorphism $\nu_i$ of $E(K_i)$
realising the positive amphicheirality of $K_i$ 
is not necessarily an involution nor a periodic map.
Even if $\nu_i$ is an involution, 
its restriction to $T_i$ is not strongly equivalent to that of $\gamma_1$.
In this case, the square of the diffeomorphism obtained by gluing $\gamma_1$ and $\nu_i$
is a non-trivial Dehn-twist along the
JSJ torus $T_i$.
As a consequence, the extension of $\gamma_1$ can never be periodic.
\end{remark}


\section{Some explicit examples of links providing admissible roots}
\label{section:roots}


To complete the proof of the existence of prime amphicheiral knots with free
period $2$ we still need to produce a link satisfying the requirements of
Definition~\ref{d:root}.

In this section we shall define three links, 
$L_{\mu}=K_0\cup \OO_{\mu}$ with $\mu=2,3,6$, 
and show that they provide admissible roots.
The three different links will allow us to produce prime amphicheiral knots
with different properties.


\subsection{The link $L_2$}
\label{ss:l2}


\begin{figure}[h]
\begin{center}
 {
\includegraphics[height=11cm]{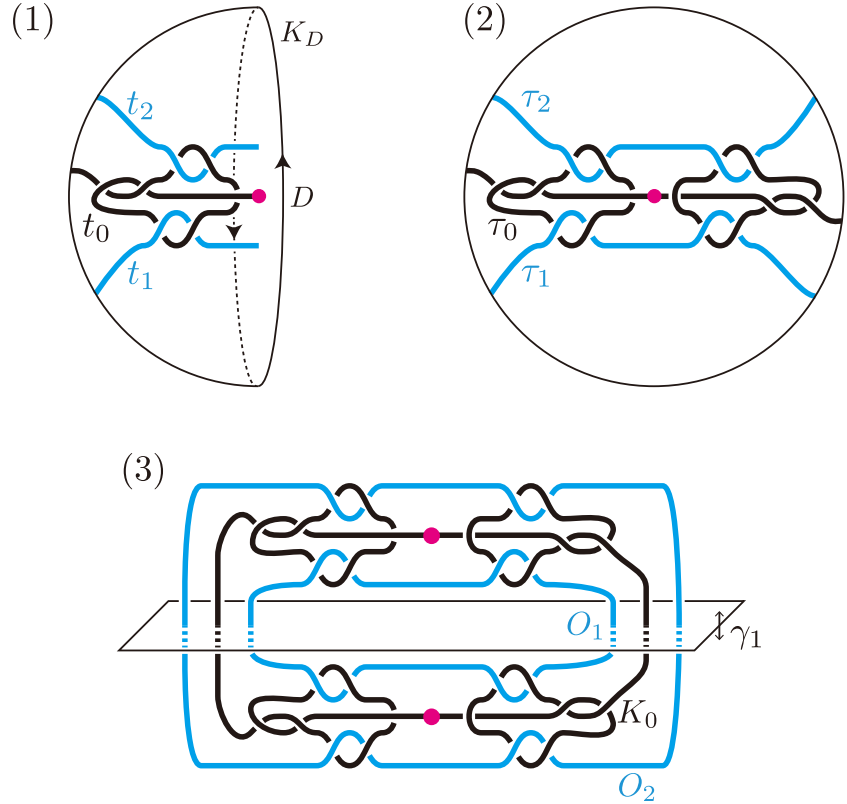}
 }
\end{center}
\caption{(1) The three-string tangle inside the half-ball fundamental domain (top
left), (2) the first symmetrisation with respect to the antipodal map (top
right), and (3) the resulting link $L_2=K_0\cup O_1\cup O_2$ 
after symmetrising with respect to the
inversion in a sphere.}
\label{fig:root}
\end{figure}

Consider the link depicted 
in Figure~\ref{fig:root}(3).
The link was constructed in such a way as to ensure that it admits the desired
$\G$-action in the following way.
Consider a fundamental domain for the $\G$-action, define a tangle
inside the domain, and get a link by symmetrising the domain and the
tangle it contains thanks to the $\G$-action.

It is not difficult to see that $\G$ has a fundamental domain which consists in
``half a ball": 
indeed each of the $3$-balls bounded by the $2$-sphere $\Fix(\gamma_1)$
forms a $\gamma_2$-invariant fundamental domain for $\langle \gamma_1\rangle$,
and the restriction of the antipodal map $\gamma_2$ to each of the $3$-ball
has half the ball as a fundamental domain.
To be even more precise, the
half ball can be considered as a cone on a closed disc. The identification
induced by the action on its boundary, which consists in gluing the boundary
circles of the discs via a rotation of order $2$, allows to obtain the global 
quotient which is a cone on a projective plane: the vertex of the cone is a 
singular point of order two (image of the fixed-points of $\gamma_2$), and the 
base of the cone is a silvered projective plane (image of the fixed-points of 
$\gamma_1$).

We define a three-component tangle inside the half ball as shown in Figure~\ref{fig:root}(1). 
One component is a ``trefoil arc" (i.e. a
trefoil knot cut at one point): this component will give rise to $K_0$ which
is $\G$-invariant, and thus contains the vertex of the cone. The other two
components are unknotted arcs which are entangled with the first component. The
procedure and result of symmetrising the tangle 
are shown in
Figure~\ref{fig:root}(2) and (3). Because of its very construction it is now clear that
$L_2$ has the required $\G$-action 
where $K_0$ is $\G$-invariant. 
It is also clear
from the picture that the components $O_1$ and $O_2$ of $L_2$ form a trivial
link.

It remains to show that the link is hyperbolic. The proof of this fact is rather
technical and is given in Section~\ref{section:link2}.


\subsection{The link $L_3$}
\label{ss:l3}


\begin{figure}[h]
\begin{center}
\includegraphics[height=5cm]{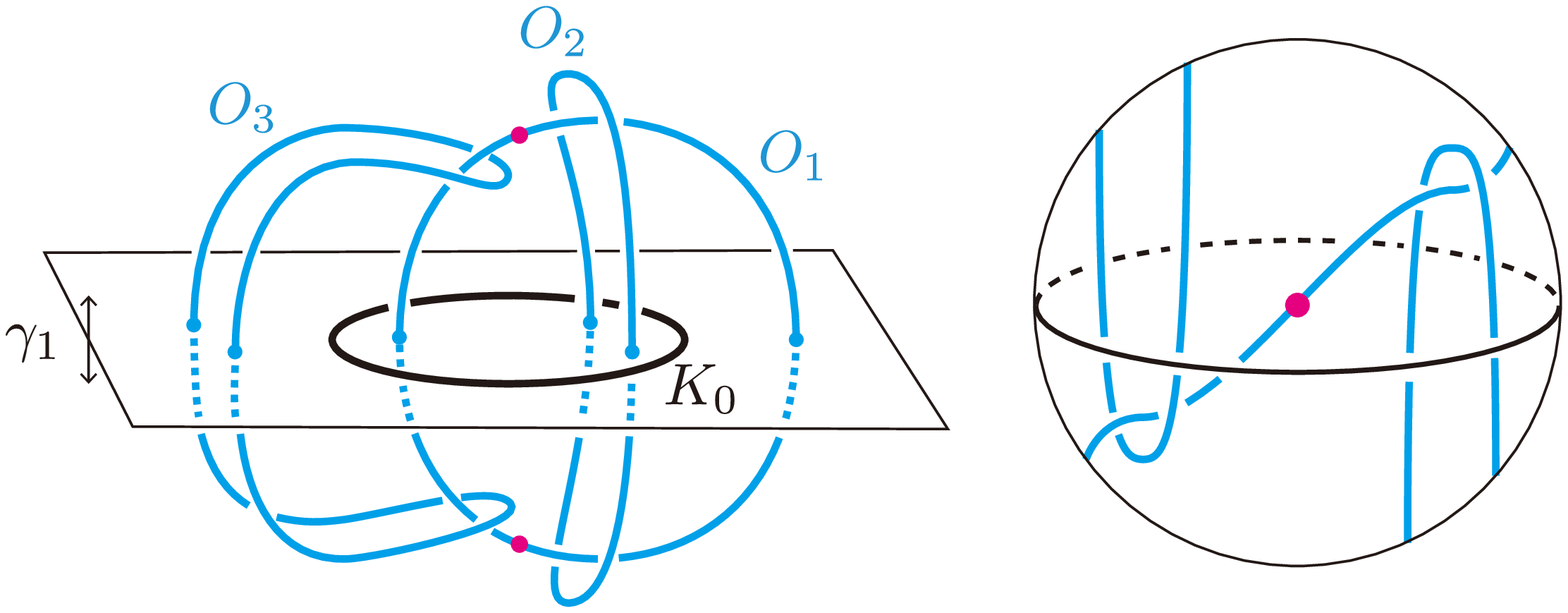}
\end{center}
\caption{The link $L_3$ pictured on the left is symmetric with respect to 
$\gamma_1$, the reflection in the horizontal plane. 
The tangle obtained as the quotient of $L_3$ by $\gamma_1$ is shown on the 
right: it is left invariant by the central symmetry of the ball,
which lifts to the involution $\gamma_2$ whose fixed point set consists of the
two 
points of $O_1$ marked in red.}
\label{fig:positive}
\end{figure}

\begin{figure}[h]
\begin{center}
\includegraphics[height=5cm]{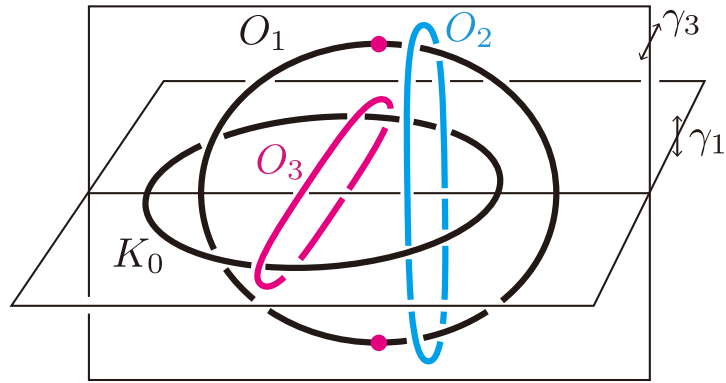}
\end{center}
\caption{An image of $L_3$ displaying the extra symmetry of the link.}
\label{fig:extra-sym}
\end{figure}

As in the previous subsection, we build a $\G$-symmetric link from a tangle in a
fundamental domain for the $\G$-action, as illustrated in
Figure~\ref{fig:positive}.

The resulting link 
$L_3$ and its component $K_0$ are $\G$-invariant by
construction. It is clear from the picture that $O_1\cup O_2\cup O_3$ is a
trivial link. Hyperbolicity of 
$L_3$ is again rather technical to establish and
will be shown in Section~\ref{section:link3}.

We remark that the symmetry group of $L_3$ is larger than $\G$. In fact, it
contains another reflection $\gamma_3$ in a $2$-sphere 
(see Figure~\ref{fig:extra-sym}). As a consequence, $L_3$ is invariant by the 
action of another Klein four-group 
$\G'=\langle \gamma_1',\gamma_2' \rangle = r^{-1}\G r$, 
where $r$ is the $\pi/2$-rotation about $Fix(\gamma_1)\cap Fix(\gamma_3)$,
and $\gamma_1'=r^{-1}\gamma_1 r =\gamma_3$ and $\gamma_2'=r^{-1}\gamma_2 r=r_1r_2r_3$.
Notice that the element 
$h:=r^2=\gamma_1\gamma_3$
is a $\pi$-rotation
acting as a strong inversion of both $K_0$ and $O_1$. (Here a \emph{strong
inversion} of a knot is an orientation-preserving smooth involution of
$\threesphere$ preserving the knot whose fixed-point set is a circle
intersecting the knot transversely in two points.)


\subsection{The link $L_6$}
\label{ss:l6}


In this case, because of the relatively large number of components, the tangle
obtained by intersecting the link with a fundamental domain is harder to
visualise. Instead of exhibiting the tangle, 
we give in Figure~\ref{fig:link3}
two pictures of the link showing that the link is symmetric with respect to
both $\gamma_1$ and $\gamma_2$. Remark that $L_6$ is a highly symmetric link, namely
it is symmetric with respect to three more reflections in vertical planes
perpendicular to $\Fix(\gamma_1)$, that is the plane of projection  
in Figure~\ref{fig:link3}.
Observe that the product of two reflections in these vertical planes
results in a $2\pi/3$-rotation, while the product of one of these reflections
with $\gamma_1$ is a $\pi$-rotation acting as a strong inversion of $K_0$.

\begin{figure}[h]
\begin{center}
 {
\includegraphics[height=5cm]{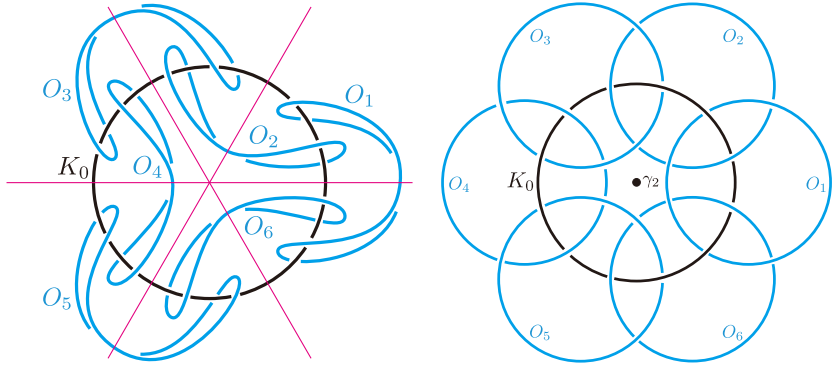}
 }
\end{center}
\caption{Two diagrams of the link $L_6$ showing the presence of different
symmetries. On the left, the link is symmetric with respect to the plane of
projection: this symmetry corresponds to $\gamma_1$. The pink lines are the
axis of the $\pi$-rotations acting as a strong inversions on $K_0$: they
coincide with the intersection of $\Fix(\gamma_1)$ with the fixed-point sets of
the other reflections that leave $L_6$ invariant. On the right, the link is
symmetric with respect to a central reflection corresponding to $\gamma_2$.}
\label{fig:link3}
\end{figure}

Once more, we postpone the proof that $L_6$ is hyperbolic to
Section~\ref{section:link6}, where we will also see that this example can be
generalised to give an infinite family of links providing admissible roots.


\section{Structure of prime amphicheiral knots with free period $2$}
\label{section:uniqueness}

In this section, we show that any prime amphicheiral knot with free period $2$
is constructed as in Section~\ref{section:knot}.

\begin{Theorem}
\label{prop:structure}
Let $K$ be a prime amphicheiral knot with free period $2$.
Then there is a 
link $L=K_0\cup \OO_{\mu}$ 
which provides an admissible root and satisfies the following conditions.
\begin{enumerate}
\item
There are non-trivial knots $K_i$ ($i=1,\cdots, \mu$) such that
\[
(\threesphere,K)=(E(\OO_{\mu}), K_0)\cup ( \cup_{i=1}^{\mu} (E(K_i), \emptyset)).
\]
Here, if a component $O_i$ of $\OO_{\mu}$ is stabilised by $\gamma_j$ for some $j=1$ or $2$,
then the knot $K_i$ is negative-amphicheiral or positive-amphicheiral 
according to whether $\gamma_j$ preserves or reverses a fixed orientation of $O_i$.
\item
For each element $g$ of $\G$, its restriction to $(E(\OO_{\mu}), K_0)$ extends to a diffeomorphism of $(\threesphere,K)$,
which we call an extension of $g$.
Moreover some extension of $f=\gamma_1\gamma_2\in \G$ is a free involution on $\threesphere$. 
Furthermore, if $K$ is $\epsilon$-amphicheiral,
then an extension of $\gamma_1$ or $\gamma_2$ realises 
the $\epsilon$-amphicheirality of $K$.
\end{enumerate}
\end{Theorem}

Let $K$ be a prime amphicheiral knot with free period $2$.
Then as already observed, it follows from \cite{S2} that
the exterior $E(K)$ of $K$ admits a nontrivial JSJ decomposition.
Let $E_0$ be the root of the decomposition, i.e.
the geometric piece containing the boundary, and
let $E_i$ ($1\le i\le \mu$) be the closure of the components
of $E(K)\setminus E_0$.
Then $E_0$ is identified with the exterior of a link $L=K_0\cup \OO_{\mu}$
where $\OO_{\mu}=\cup_{i=1}^{\mu}O_i$ is a $\mu$-component trivial link,
and $E_i$ is identified with a knot exterior $E(K_i)$ 
for each $i$ with $1\le i\le \mu$ 
(see e.g. \cite[Lemma 2.1]{S1}).

We show that $L$ provides an admissible root in the sense of Definition \ref{d:root},
namely, $E_0$ is hyperbolic
and, after an isotopy, both $L$ and $K_0$ are $\G$-invariant.

We first prove that the root $E_0$ is hyperbolic.
Otherwise, $E_0$ is a Seifert fibered space embedded in $E(K)$
with at least two (incompressible) boundary components and so
$E_0$ is a composing space or a cable space (see 
\cite[Lemma VI.3.4]{JS} or \cite[Lemma IX.22]{J}).
Since $K$ is prime, $E_0$ is not a composing space.
So $E_0$ is a cable space and hence $K$ is a cable knot.
However, a cable knot cannot be amphicheiral, a contradiction.
Though this fact should be well known,
we could not find a reference, so we include a proof for completeness.

\begin{Lemma}
A cable knot is not amphicheiral.
\end{Lemma}

\begin{proof}
Let $K$ be a $(p,q)$-cable of some knot,
where $p$ is a positive integer greater than $1$
and $q$ is an integer relatively prime to $p$.
Then the root $E_0$ of the JSJ decomposition of $E(K)$ is
the Seifert fibered space with base orbifold an annulus with one cone point,
such that the singular fiber has index $(p,q)$.
If $K$ is amphicheiral,
then $E_0$ admits an orientation-reversing diffeomorphism, $\gamma$,
and we may assume that $\gamma$ preserves the Seifert fibration
(see e.g. \cite[Theorem 3.9]{Scott}).
Hence we have $q/p\equiv -q/p \in \Q/\Z$, and so $p=2$.
Thus $E_0$ is identified with the exterior of the pretzel link $P(2,-2,q)$.
Since $\gamma$ is a restriction of a diffeomorphism of $(\threesphere,K)$ to
$E_0$, it extends to a diffeomorphism of $(\threesphere,P(2,-2,q))$, which
reverse the orientation of $\threesphere$.
This contradicts the fact that $P(2,-2,q)$ is not amphicheiral.
Here the last fact can be seen, for example,
by using \cite[Theorem 4.1]{S3}.
\end{proof}

Let $\Isom^*(E_0)$ be the subgroup of the isometry group
of the complete hyperbolic manifold $E_0$
consisting of those elements, $g$,
which extend to a diffeomorphism of $(\threesphere,K)$. 
(To be precise, we identify $E_0$ with the non-cuspidal part
of a complete hyperbolic manifold.)
Denote by
$\Isom^*_+(E_0)$ the subgroup of $\Isom^*(E_0)$
consisting of elements whose extensions to $(\threesphere,K)$
preserve the orientation of both $\threesphere$ and $K$.
Then we have the following lemma,
which holds a key to
the main result in this section.
Thus we include a proof,
even though it follows from
\cite[the last part of the proof of Lemma 2.2]{S1}.

\begin{Lemma}
\label{lem:cyclic}
{\rm (1)} The action of  $\Isom^*(E_0)$ on $E_0$ extends to 
a smooth action on $(\threesphere,L)$.

{\rm (2)} $\Isom^*_+(E_0)$ is a finite cyclic group.
\end{Lemma}
\begin{proof}
(1) 
Let $\ell_i$ and $m_i$, respectively, be the longitude and the meridian
of the knot $K_i$ with $E_i=E(K_i)$ ($1\le i\le \mu$). 
Recall that they are the meridian and the longitude of $O_i$, respectively.
Since each element of $g\in \Isom^*(E_0)$ extends to a diffeomorphism
of $(\threesphere,K)$, it follows that
if $g(E_i)=E_j$ then $g(\ell_i)=\pm \ell_j$ and so
$g$ maps the meridian of $O_i$ to the meridian
(possibly with reversed orientation) of $O_j$.
Hence the action of  $\Isom^*(E_0)$ on $E_0$ extends to an action on $(\threesphere,L)$.

(2)
If $\Isom^*_+(E_0)$ is not cyclic, then
the restriction of the extended action to $K_0$ is not effective,
i.e., there is a nontrivial element, $g$,
of $\Isom^*_+(E_0)$
whose extension, $\bar g$, to $(\threesphere, K_0\cup \OO_{\mu})$
is a nontrivial periodic map with $\Fix(\bar g)=K_0$.
By the positive solution of the Smith conjecture \cite{MB},
$\Fix(\bar g)$ is a trivial knot, and
$\bar g$ gives a cyclic periodicity of the trivial link $\OO_{\mu}$.
By \cite[Theorem 1]{S0}, such periodic maps are \lq\lq standard'',
and so there are mutually disjoint discs
$D_i$ ($1\le i\le \mu$)
in $\threesphere$ with $\partial D_i=O_i$
such that, for each $i\in\{1,\cdots,\mu\}$,
either (i) $\bar g(D_i)=D_i$, so that $\Fix(\bar g)$ intersects $D_i$
transversely in a single point,
or (ii) $\bar g(D_i)=D_j$ for some $j\ne i$.
If (ii) happens for some $i$, then $D_i$ is disjoint from $\Fix(\bar g)=K_0$
and $\OO_{\mu}-O_i$; so the torus $T_i$ is compressible in $E_0$, a contradiction.
Hence the link $K_0\cup\OO_{\mu}$ is as illustrated in
Figure~\ref{fig:hopf} and therefore $E_0$ is a composing space,
a contradiction.
\end{proof}

\begin{remark}\label{r:effective}
We wish to stress that the proof of Lemma~\ref{lem:cyclic} implies that
the extension of $\Isom^*_+(E_0)$ to $(\threesphere, L)$ must act effectively
on $K_0$.
\end{remark}

\begin{figure}[h]
\begin{center}
 {
  \includegraphics[height=3cm]{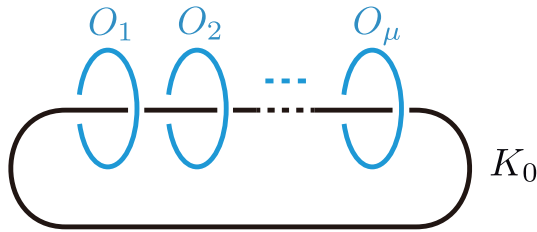}
 }
\end{center}
\caption{A connected sum of Hopf links.}
\label{fig:hopf}
\end{figure}

Let $\gamma$ be a diffeomorphism of $(\threesphere,K)$ which realises
the amphicheirality of $K$, and let $f$ be a free involution of
$(\threesphere,K)$. After an isotopy $\gamma$ preserves $E_0$ and so determines
a self-diffeomorphism of $E_0$. We will denote by $\check\gamma$ the
orientation-reversing isometry of $E_0$ isotopic to this diffeomorphism.
By \cite{BF, S1}, we may assume that the involution $f$ restricts to an
isometry $\check f$ of $E_0$.
We denote by
$\bar\gamma$ and $\bar f$
the periodic diffeomorphisms of $(\threesphere, L)$
obtained as extensions of $\check\gamma$ and $\check f$,
respectively, whose existence is guaranteed by Lemma~\ref{lem:cyclic}.

\begin{Lemma}
\label{lemma:free}
The diffeomorphism $\bar f$ is a free involution of $\threesphere$.
\end{Lemma}

\begin{proof}
By construction, 
$\bar f$ is an involution which acts freely on $E_0$. 
If $\bar f$ has a fixed point,
it must occur inside a regular neighbourhood of some component $O_i$. Moreover,
the fixed-point set must coincide with an $O_i$ which is $\bar f$-invariant. It
follows that the slope of $\bar f$ along $O_i$ is the meridian. 
Since the
meridian of $O_i$ coincides with the longitude of $E_i=E(K_i)$,
the action of $f$ on $E(K_i)$ cannot be free by \cite[Lemma 1.2(3)]{S1},
against the assumption.
\end{proof}

\begin{Lemma}
\label{lem:order-2}
We can choose the diffeomorphism $\gamma$ of $(\threesphere,K)$ giving
amphicheirality of $K$ so that the corresponding isometry 
$\check\gamma$ of $E_0$ satisfies the condition that the subgroup 
$\langle \check\gamma, \check f\rangle$ of $\Isom(E_0)$ is isomorphic to 
$(\Z/2\Z)^2$. Moreover, the new $\gamma$ preserves or reverses the orientation 
of $K$ according to whether the original $\gamma$ preserves or reverses the 
orientation of $K$.
\end{Lemma}

\begin{proof}
Since $\Isom(E_0)$ is the isometry group of a hyperbolic manifold with finite
volume, it is a finite group. The element $\check\gamma$ reverses the 
orientation of the manifold, so it must have even order and, up to taking an 
odd power, we can assume that its order is a power of $2$. The cyclic subgroup
$\Isom^*_+(E_0)$ of $\Isom^*(E_0)$ is clearly normal. It contains the subgroup
generated by $\check f$, which is its only subgroup of order $2$. It follows 
that the subgroup generated by $\check f$ is normalised by $\check \gamma$, and 
hence $\check f$ and $\check \gamma$ commute. If the order of $\check\gamma$ is 
$2$ we are done. Otherwise, $\check \gamma^2$ belongs to 
the cyclic group
$\Isom^*_+(E_0)$, and $\check f$ is a power of $\check \gamma$. 
Note that the periodic map $\bar\gamma$ of 
$(\threesphere,L)$
obtained as
the extension of $\check\gamma$ reverses the orientation of $\threesphere$
and hence it has a nonempty fixed point set.
Thus $\bar\gamma^2$ also has a nonempty fixed point set, and so does $\bar f$,
a contradiction. Hence $\langle \check\gamma, \check f\rangle\cong (\Z/2\Z)^2$.
The last assertion is obvious from the construction.
\end{proof}

Consider the subgroup 
$\langle \check\gamma, \check f\rangle$
of $\Isom(E_0)$ in the above lemma,
and let $\langle \bar f, \bar\gamma\rangle$ be 
its extension to a group action on $(\threesphere, L)$.
The main result of \cite{DL} implies that
$\langle \bar f, \bar\gamma\rangle$ is conjugate,
as a subgroup of $\Diff(S^3)$, 
to a subgroup 
$\bar\G$ of the orthogonal group 
$\mathrm{O}(4) \cong\mathrm{Isom}(\threesphere)$.
By using the facts that $\bar f$ is a free involution
and that $\bar \gamma$ descends to an orientation-reversing involution on 
$S^3/\bar f$, we may assume $\bar\G$ is equal to the group 
$\G=\langle \gamma_1,\gamma_2\rangle$ (cf. \cite{K}).
Thus we may assume $L=K_0\cup\OO_{\mu}$ is $\G$-invariant,
and it provides an admissible root in the sense of Definition~\ref{d:root}.
The remaining assertions of Theorem \ref{prop:structure} follow from
the arguments in Section~\ref{section:knot}. 


\section{Proofs of Theorems~\ref{thm1}, \ref{thm2}, and \ref{thm:minimal}}
\label{section:thm2}


Let $K$ be a prime amphicheiral knot with free period $2$.
Then by Theorem \ref{prop:structure}, 
there is a link $L=K_0\cup \OO_{\mu}$ 
which provides an admissible root and
non-trivial knots $K_i$ ($i=1,\cdots, \mu$) such that
\[
(\threesphere,K)=(E(\OO_{\mu}), K_0)\cup ( \cup_{i=1}^{\mu} (E(K_i), \emptyset)).
\]
In particular, $L$ is $\G$-invariant and the root $E_0$ of the JSJ decomposition of $E(K)$
is identified with $E(L)=E(K_0\cup\OO_{\mu})$ where $\mu=\mu(K)$.

We will start with the proof of Theorem~\ref{thm2}: each point of the theorem
will be proved as a claim in this section.

\begin{Claim}
\label{c:inv-comp}
$L$ contains at most two components whose stabiliser is the whole group
$\G$. These components are either a great circle in the $2$-sphere
$\Fix(\gamma_1)$, or a knot, either trivial or composite, that meets
$\Fix(\gamma_1)$ and contains the two fixed points of $\gamma_2$.
\end{Claim}

Here by a \emph{great circle} we mean an embedded circle in the $2$-sphere
$\Fix(\gamma_1)$ which is invariant by the antipodal map induced by $f$ on
$\Fix(\gamma_1)$.

\begin{proof}
Let $L_i$ be a component which is left invariant by $\G$. If $L_i$ is contained
in $\Fix(\gamma_1)$ then it must be a trivial knot which is a great circle of
$\Fix(\gamma_1)$. 
Since the $2$-sphere $\Fix(\gamma_1)$ cannot contain two mutually disjoint great circles,
no other component contained in $\Fix(\gamma_1)$ can
have $\G$ as stabiliser. Assume now that $L_i$ is not contained in
$\Fix(\gamma_1)$. Since its stabiliser is $\G$, $L_i$ cannot be contained in one
of the two balls of $\threesphere\setminus \Fix(\gamma_1)$. As a consequence, it
intersects $\Fix(\gamma_1)$ transversally in two antipodal points, and each of
the two balls of $\threesphere\setminus \Fix(\gamma_1)$ in an arc. Since each of
these arcs is left invariant by $\gamma_2$, it must contain one of the two
fixed points of $\gamma_2$. Of course, at most one component of $L$ can contain
$\Fix(\gamma_2)$. 
Since the two arcs are exchanged by $\gamma_1$, it follows that they are both 
unknotted, in which case $L_i$ is trivial, or both knotted, in which case $L_i$ 
is composite.
\end{proof}

\begin{Claim}
\label{c:atleast2}
$L$ contains at least one pair of components of $\OO_{\mu}$, 
such that (i) each of the components intersects the $2$-sphere $\Fix(\gamma_1)$ 
transversely in two points, 
(ii) the 
stabiliser of each of the components is generated by $\gamma_1$,
and (iii) $f$ interchanges the two components. 
\end{Claim}

\begin{proof}
Since $L$ is hyperbolic, we can assume that $\gamma_1$ acts as a hyperbolic
isometry on its complement. It follows that each component of
$\Fix(\gamma_1)\setminus L$ is a totally geodesic surface in 
the hyperbolic manifold $\threesphere -L$. 
Thus no component of $\Fix(\gamma_1)\setminus L$ is
a sphere, a disc, or an annulus.
On the other hand, if a component of $L$ is not disjoint from the $2$-sphere $\Fix(\gamma_1)$, 
then it is either contained in $\Fix(\gamma_1)$ or
intersects $\Fix(\gamma_1)$ transversely in two points.
(In fact, if a component of $L$ intersects transversely $\Fix(\gamma_1)$ at a point,
then it is $\gamma_1$-invariant and so it intersects 
$\Fix(\gamma_1)$ transversely in precisely two points.)
Now the claim follows from the fact that either (i) $\Fix(\gamma_1)$ contains 
no component of $L$, so that $\Fix(\gamma_1)\setminus L$ is a punctured sphere 
with at least three (actually four) punctures, or (ii) contains some components
of $L$ in which case at least two components of $\Fix(\gamma_1)\setminus L$ are
punctured discs with at least two punctures.
\end{proof}

\begin{Claim}
\label{c:No-f}
$L$ contains no component with stabiliser generated by $f$.
\end{Claim}

This was proved in Lemma~\ref{lem:no-f-stab}.

\begin{Claim}
\label{c:K0}
\rm{(1)}
If $K$ is positive-amphicheiral, then $K_0$ must be contained in 
$\Fix(\gamma_1)$.

\rm{(2)}
If $K$ is negative-amphicheiral, then $K_0$ must contain $\Fix(\gamma_2)$ and 
intersect transversally $\Fix(\gamma_1)$ in two points.
\end{Claim}

\begin{proof}
By assumption, $K_0$ is $\G$-invariant, so it must be as described in
Claim~\ref{c:inv-comp}. 
Consider the action induced by $\G$ on $H_1(K_0;\Z)$.
Since $f\in \G$ is 
a free involution of $\threesphere$,
the action of $f$ on $H_1(K_0;\Z)$ is trivial, 
so the actions of $\gamma_1$ and $\gamma_2$ on it coincide.
It is now easy to see that 
(i) if $K_0$ is contained in $\Fix(\gamma_1)$ then $\gamma_1$ acts trivially on $H_1(K_0;\Z)$,
and 
(ii) if $K_0$ meets $\Fix(\gamma_1)$ transversally
then $\gamma_1$ acts on $H_1(K_0;\Z)$ as multiplication by $-1$. 
Since $\epsilon$-amphicheirality of $K$ is realised by an extension of 
$\gamma_1$ or $\gamma_2$ by Theorem \ref{prop:structure}, we obtain the desired 
result.
\end{proof}

This ends the proof of Theorem~\ref{thm2}.

\medskip

We now explain Remark \ref{rem:invertible}. Suppose that the prime knot $K$ 
with free period $2$ is both positive- and negative-amphicheiral. Since 
$\Isom^*_+(E_0)$ is a finite cyclic group by Lemma~\ref{lem:cyclic}(2), there 
is a unique element $f \in \Isom^*(E_0)$ which extends to a smooth involution 
of $(S^3,K)$ realising the free period $2$. For $\epsilon\in\{+,-\}$, let 
$\gamma_{\epsilon}$ be the order $2$ element of $\Isom^*(E_0)$ which realises 
the $\epsilon$-amphicheirality of $K$, such that 
$\langle f, \gamma_{\epsilon}\rangle\cong (\Z/2\Z)^2$ 
(cf. Lemma \ref{lem:order-2}). Now recall that the finite group $\Isom^*(E_0)$ 
extends to an action of $(\threesphere, L)$ by Lemma~\ref{lem:cyclic}, and so 
we identify it with a finite 
subgroup of $\Diff(\threesphere, L)$. Then 
$\gamma_+\gamma_-$ preserves the orientation of $\threesphere$ and reverses the 
orientation of the component $K_0$, and so it realises the invertibility of 
$K_0$. Since $(\gamma_+\gamma_-)^2$ acts on $K_0$ as the identity map,
the periodic map $(\gamma_+\gamma_-)^2$ must be the identity map according to 
Remark~\ref{r:effective}. Thus $(\gamma_+\gamma_-)^2=1$ in $\Isom^*(E_0)$.
Hence we see $\langle f, \gamma_+, \gamma_-\rangle \cong (\Z/2\Z)^3$.
The main result of \cite{DL} guarantees that, as a subgroup of 
$\Isom(\threesphere)$, this group is smoothly conjugate to a subgroup of 
$\Isom(\threesphere)$. Remark \ref{rem:invertible} now follows from this fact.

\medskip

We need the following lemma in the proof of Theorem~\ref{thm:minimal}.

\begin{Lemma}
\label{p:specialcase}
Let $L=K_0\cup\OO_2$ be a three component link providing an admissible root. 
Then $K_0$ cannot be contained in the $2$-sphere $\Fix(\gamma_1)$.
\end{Lemma}

\begin{proof}
Assume by contradiction that $L$ is a link with three components providing an
admissible root and such that $K_0$ is contained in $S:=\Fix(\gamma_1)$. 
Recall that, because of Claim~\ref{c:atleast2}, each component of $\OO_2$ must 
intersect $S$ transversally in two points. Thus $S$ gives a $2$-bridge 
decomposition of $\OO_2$. (For terminology and standard facts 
on $2$-bridge 
links, we refer to \cite[Section 2]{LS}.) By the uniqueness of the $2$-bridge 
spheres or by the classification of $2$-bridge spheres, $S$ is identified with 
the standard $2$-bridge sphere of the $2$-bridge link of slope $1/0$. The knot 
$K_0$ is an essential simple loop on the $4$-times punctured sphere 
$S \setminus \OO_2$, and the isotopy type of 
any such loop is completely determined by its slope $s\in \Q\cup\{1/0\}$.
We can easily check that the involution $\gamma_2$ sends a loop of slope $s$ to 
a loop of a slope of slope $-s$. Hence the slope of $K_0$ is either $0$ or 
$1/0$. According to whether the slope is $0$ or $1/0$, the link 
$L=K_0\cup \OO_2$ is the connected sum of two Hopf links or $3$-component 
trivial link, a contradiction.
\end{proof}

We shall now give the proof of Theorem~\ref{thm:minimal}. 
Claim~\ref{c:atleast2} shows that $\mu_{-}, \mu_{+}\ge 2$. The link $L_2$
defined in Section~\ref{section:roots}, on the other hand, shows that
$\mu_{-}\le 2$. In fact, both $\gamma_1$ and $\gamma_2$ act on the component 
$K_0$ by reversing its orientation, they extend to diffeomorphisms of 
$(\threesphere,K)$ which give negative-amphicheirality of the knot $K$.
The first part of the theorem now follows from Claims~\ref{c:atleast2} and 
\ref{c:K0}.

For the second part, once again the link $L_3$ defined in 
Section~\ref{section:roots} shows that $\mu_{+}\le 3$. 
Lemma~\ref{p:specialcase} assures that $\mu_{+}= 3$. In fact, both $\gamma_1$ 
and $\gamma_2$ act on the component $K_0$ by preserving its orientation, they 
extend to diffeomorphisms of $(\threesphere,K)$ which give 
positive-amphicheirality of the knot $K$. Now the second part of the theorem 
follows again from Claims~\ref{c:atleast2} and \ref{c:K0}. 

\medskip

We pass now to the proof of Theorem~\ref{thm1}. Assuming the hyperbolicity of 
the links $L_{\mu}$ with $\mu=2,3,6$, which is proved in 
Sections~\ref{section:link2}, \ref{section:link3} and \ref{section:link6},
the existence of prime, amphicheiral knots with free period $2$ was established 
in Section~\ref{section:roots}. 
To finish the proof of Theorem~\ref{thm1}, we 
only need to check that we can find prime amphicheiral knots admitting free 
period $2$ that are:
\begin{enumerate}
\item negative-amphicheiral but not positive-amphicheiral;
\item positive-amphicheiral but not negative-amphicheiral;
\item positive- and negative-amphicheiral at the same time, that is amphicheiral
and invertible.
\end{enumerate}

We show why these statements hold by proving a series of claims.

\begin{Claim}
\label{c:neg}
The prime amphicheiral knots whose root is the exterior of $L_2$ are negative-amphicheiral and cannot be positive-amphicheiral.
\end{Claim}

\begin{proof}
Since $\gamma_1$ acts on $K_0$ by reversing its orientation, the extensions of 
both $\gamma_1$ and $\gamma_2$ act by inverting $K$, which is thus 
negative-amphicheiral. Assume now by contradiction that $K$ is also 
positive-amphicheiral. According to Theorem~\ref{thm2}, $L_2$ must admit 
another action $\G'$ of the Klein four group, containing a reflection 
$\gamma'_1$ in a $2$-sphere $\Fix(\gamma'_1)$, such that $K_0$ is contained in 
$\Fix(\gamma'_1)$. This is however impossible for $K_0$ is not trivial.
\end{proof}

\begin{Claim}
\label{c:pos}
The prime amphicheiral knots whose root is the exterior of $L_3$ are 
positive-amphicheiral. 
If the knots $K_i$, $i=1,2,3$, are all positive-amphicheiral but none of them 
is negative-amphicheiral, then $K$ itself is not negative-amphicheiral.
\end{Claim}

\begin{proof}
The fact that $K$ is positive-amphicheiral can be seen as in 
Claim \ref{c:neg}.
Observe that $\gamma_1$ acts on each $O_i$ by reversing its orientation. As a
consequence, each $E(K_i)$ must be the exterior of a positive-amphicheiral
knot, i.e., each $K_i$ must be positive-amphicheiral. We assume now that each 
$K_i$ is not negative-amphicheiral. Suppose by contradiction that $K$ is 
negative-amphicheiral. Then, by Theorem~\ref{thm2}, $L_3$ admits another action 
$\G'=\langle \gamma_1',\gamma_2'\rangle$ of the Klein four-group such that 
$K_0$ contains $\Fix(\gamma_2')\cong \zerosphere$. Since $\OO_3$ has three
components, precisely one of them, say $O_j$, for a $j\in\{1,2,3\}$, must be 
$\G'$-invariant and, according to Claim~\ref{c:inv-comp}, it must be contained 
in the $2$-sphere $\Fix(\gamma_2')$. 
Since $\gamma_1'$ acts trivially on the first integral homology groups of such components,
$K_j$ must be negative-amphicheiral against the assumption.
\end{proof}

\begin{Claim}
\label{c:invertible}
Let $K$ be a prime amphicheiral knot admitting free period $2$ whose root is 
the exterior of $L_{\mu}$ with $\mu=3$ or $6$, that is constructed as in 
Section \ref{section:knot} by using the symmetry $\G$.

{\rm (1)}
If $\mu=3$ and $K_1$ is invertible, then $K$ is invertible.

{\rm (2)}
If $\mu=6$ and all $K_i$ are copies of the same negative-amphicheiral knot, 
then $K$ is invertible.
\end{Claim}

\begin{proof}
(1) Consider the link $L_3$. Then, by construction of the knot $K$, $K_1$ is a 
positive-amphicheiral knot admitting free period $2$, and $K_2$ and $K_3$ are 
copies of a positive-amphicheiral knot. As noted in Subsection \ref{ss:l3}, $L$ 
is invariant by the action of the group
$\hat \G=\langle \gamma_1,\gamma_2,\gamma_3\rangle <\Isom(\threesphere)$,
where $\gamma_i$ are 
orientation-reversing involutions
as illustrated in Figure \ref{fig:extra-sym}.
Note that $\hat \G$ is the direct product of the group 
$\G=\langle \gamma_1,\gamma_2\rangle$ and the order $2$ cyclic group generated 
by $h:=\gamma_1\gamma_3$. Then $h$ reverses the orientations of $K_0$ and $O_1$,
and preserves the orientations of $O_2$ and $O_3$. 
The stabiliser in $\hat \G$ 
of $O_j$ is equal to $\hat \G$ or $\langle \gamma_1, h\rangle$ according to 
whether $j=1$ or $j\in \{2,3\}$. Since $h$ acts on $H_1(T_j;\Z)$ ($j=2,3$) 
trivially, there is no obstruction in extending $h$ to $E(K_j)$ ($j=2,3$).
On the other hand, $h$ acts on $H_1(T_1;\Z)$ as $-I$. Thus if $K_1$ is 
invertible, then $h$ extends to a diffeomorphism of $(\threesphere,K)$. The 
extended $h$ gives invertibility of $K$.

(2) Consider the link $L_6$. Then $L_6$ is invariant by the group 
$\tilde \G=\langle \gamma_1,\gamma_2,\gamma_3\rangle <\Isom(\threesphere)$,
where $\gamma_3$ is the reflection in the vertical plane intersecting the 
projection plane in the horizontal pink line in Figure \ref{fig:link3}.
If all $K_i$ are copies of the same negative-amphicheiral knot, then each of 
$\gamma_3$ and $\rho$ extends to a diffeomorphism of $(\threesphere,K)$, and 
hence every element of $\tilde \G$ extends to a diffeomorphism of 
$(\threesphere,K)$. We can observe that any extension of $h:=\gamma_1\gamma_3$
realises the invertibility of $K$.
\end{proof}

The proof of Theorem~\ref{thm1} is now complete.

\begin{remark}
\label{r:strongly-invertible}
(1)
In the situation described in Claim~\ref{c:invertible}, the involution
$h=\gamma_1\gamma_3$ extends to a diffeomorphism $\hat h$ of $(\threesphere, K)$
which realises the invertibility of the knot. However, any extension 
$\hat h$ of $h$ (to be precise, the restriction of $h$ to $E(L_{\mu})$) to 
$(\threesphere, K)$ cannot be an involution. We explain this in the case where
the root of $K$ is the exterior of $L_3$. In this case $h$ stabilises all 
components of $L_3$, and acts as a $\pi$-rotation on each of the components 
$O_2$ and $O_3$. Thus, for $i=2,3$, the restriction of $h$ to 
$T_i=\partial N(O_i)$ is an orientation-preserving free involution whose slope 
is the longitude of $O_i$. This means that $h$ acts on $E(K_i)$, $i=2,3$
by preserving the meridian of $E(K_i)$. The positive solution to the Smith 
conjecture implies that $\hat h$ cannot have finite order. The same argument 
works for the case when the root of $K$ is the exterior of $L_6$.

(2)
Consider now the element $\gamma_2\gamma'_1=fh$. If the root of $K$ is the
exterior of $L_6$, then $fh$ extends to a strong inversion of $K$. 
To see this, it suffices to observe that this element does not
stabilise any component of $L_6$ other than $K_0$. In the case where the root
of $K$ is the exterior of $L_3$, $fh$ is a strong inversion provided that
the knot $K_1$
is strongly invertible, since $O_1$ meets 
$\Fix(fh)$
and is left invariant by $fh$, 
while the components $O_2$ and $O_3$ are exchanged.
\end{remark}

The following result is a consequence of the discussion in the previous section
and of Theorem~\ref{thm2}.

\begin{Proposition}
\label{c:isom}
Assume $K$ is a prime amphicheiral knot with free period $2$ and let $E_0$ be
its root which can be identified with the exterior of a link $L=K_0\cup\OO$.
Let $2n\ge2$ be the order of the cyclic group $\Isom^*_+(E_0)$. Then precisely
one of the following situations occurs:
\begin{enumerate}
\item $K$ is positive-amphicheiral but not negative-amphicheiral and
$\Isom^*(E_0)$ is isomorphic to $\Z/2n\Z\times\Z/2\Z$;
\item $K$ is negative-amphicheiral but not positive-amphicheiral and
$\Isom^*(E_0)$ is isomorphic to 
the dihedral group $\Z/2n\Z\rtimes\Z/2\Z$;
\item $K$ is invertible and $\Isom^*(E_0)$ is isomorphic to
the semi-direct product
$\Z/2n\Z\rtimes(\Z/2\Z\times\Z/2\Z)$ where one copy of $\Z/2\Z$ acts dihedrally
on $\Z/2n\Z$ and the other trivially.
\end{enumerate}
\end{Proposition}

\begin{proof}
Let $\check\gamma$ be an element in $\Isom^*(E_0)$ which reverses the
orientation of $E_0$. We saw in Lemma~\ref{lem:order-2} that, up to taking an
odd power, we can choose $\check\gamma$ to be of order $2$. This means that if
$K$ is not invertible the exact sequence
$$1\longrightarrow \Isom^*_+(E_0) \longrightarrow \Isom^*(E_0)\longrightarrow
\Z/2\Z \longrightarrow 1$$
splits and it is enough to understand how $\Z/2\Z$ acts on $ \Isom^*_+(E_0)$.
To conclude it suffices to observe that $\Isom^*(E_0)$ acts on the circle $K_0$
and the action of $\hat\gamma$ is effective if and only if it reverses the
orientation of the circle and thus acts dihedrally on $\Isom^*_+(E_0)$.

If $K$ is invertible, the argument is the same provided we can show that the
exact sequence
$$1\longrightarrow \Isom^*_+(E_0) \longrightarrow \Isom^*(E_0)\longrightarrow
\Z/2\Z \times \Z/2\Z \longrightarrow 1$$
splits again. 
This, however, follows easily from Remark \ref{rem:invertible}.
\end{proof}


\section{More information on the root $E_0=E(L)$}
\label{section:additional-information}

In this section, we give refinements of the arguments in the previous section,
and present more detailed results on the structure of the root $E_0=E(L)$.

We first give a characterisation 
of the links $L=K_0\cup\OO_{\mu}$ that provide an admissible root of a positive 
amphicheiral knot. Recall that, after an isotopy, such a link $L$ is invariant 
by the action $\G$ and $K_0$ is contained in $\Fix(\gamma_1)$.
The following proposition provides a more precise description of such links.

\begin{Proposition}\label{p:K0inS}
Let $L=K_0\cup \OO_{\mu}$ 
be a link providing an admissible root. Assume that $K_0$
is contained in $S=\Fix(\gamma_1)$. 
Then the following hold.
\begin{enumerate}
\item
Suppose no component of $\OO_{\mu}$ is $f$-invariant. Then there is an 
$f$-invariant family of pairwise disjoint discs $\{D_i\}_{i=1}^\mu$ such that 
$\partial D_i=O_i$ and $D_i$ intersects $S$ transversely 
precisely
in a single arc 
($1\le i\le \mu$).
\item
Suppose one component, say $O_1$, of $\OO_{\mu}$ is $f$-invariant.
Then there is an $f$-invariant family of discs 
$\{D_1',D_1''\}\cup\{D_i\}_{i=2}^\mu$ with disjoint interiors, such that 
$\partial D_1' =\partial D_1''=O_1$, $\partial D_i=O_i$ ($2\le i\le \mu$),
and that each disc 
in the family intersects $S$ transversely 
precisely
in a single arc.
\end{enumerate}
In particular, each component of $\OO_{\mu}$ meets 
$S$ transversally in two points, in both cases.
\end{Proposition}

\begin{proof}
(1) Suppose no component of $\OO_{\mu}$ is $f$-invariant.
Then, since $f$ is orientation-preserving and since $\OO_{\mu}$ is a trivial link, 
the equivariant Dehn's lemma \cite[Theorem 5]{MY}
implies that
there is an $f$-invariant family of mutually disjoint 
discs $\{D_i\}_{i=1}^\mu$ such that $\partial D_i=O_i$. 
To prove the proposition 
we will show that
we can choose $\{D_i\}_{i=1}^\mu$ so that
each $D_i$ intersects
$S:=\Fix(\gamma_1)$ precisely in one arc.

By small isotopy, we can assume that 
the family $\{D_i\}_{i=1}^\mu$
intersects the sphere $S$ transversally:
in the case where $O_i=\partial D_i$ is contained in $S$ 
by ``transversally along $O_i$" we mean that there is 
a collar neighbourhood of $\partial D_i$ in $D_i$ that intersects $S$ only 
along $\partial D_i$. 

We want to show that one can eliminate all 
circle components of 
$S\cap(\cup_{i=1}^\mu \interior D_i)$
We start by observing that, for each such circle component $C$, 
there is a well-defined notion of 
the inside of $C$ (in the sphere $S$).
Indeed, for each $C$ one can consider
the circle $f(C)$. 
These are two disjoint circles in $S$, so they bound disjoint
subdiscs of $S$, 
and the {\em inside} of $C$ is defined to be the interior of the 
subdisc bounded by $C$ that is disjoint from $f(C)$: 
observe that the notion of the inside is $f$-equivariant,
i.e., the inside of $f(C)$ is the image of the inside of $C$ by $f$. 
Remark now that  
a circle component of $S\cap(\cup_{i=1}^\mu \interior D_i)$ can be of two types:
either it contains some arc 
component of $S\cap(\cup_{i=1}^\mu D_i)$ 
({\em first type}) or it does not ({\em second type}).

We can eliminate all circles of the second type in an 
$f$-equivariant way, as follows.
Note that a circle 
of this type can only contain circles of the same type
in its inside.
Let $C$ be any such circle which is innermost in $S$ so that
$f(C)$ is also innermost in $S$. The circle $C$ (respectively $f(C)$) is also 
contained
in a disc $D_i$ (respectively $f(D_i)$) of our family, where it 
bounds
a subdisc. We now replace by surgery such subdisc in $D_i$ (respectively 
$f(D_i)$) with a disc parallel to the subdisc of $S$ contained in $C$ 
(respectively $f(C)$) slightly off $S$, chosen appropriately on the side of 
$S$ that allows to eliminate the intersection. Notice that this operation can 
be carried out even when $C$ (respectively $f(C)$) contains points of $K_0$ and 
it may result in eliminating other circle intersections, since $C$
(respectively $f(C)$) is not necessarily innermost in $\cup_{i=1}^\mu D_i$.

The above argument shows that we can assume that all circles in
$S\cap(\cup_{i=1}^\mu \interior D_i)$
are of the first type, i.e.,
contain arc components of $S\cap(\cup_{i=1}^\mu D_i)$. 
Under
this hypothesis, we now show how to eliminate all circles. 
Let now $C$ be a circle component of 
$S\cap(\cup_{i=1}^\mu \interior D_i)$
which is innermost in $\cup_{i=1}^\mu D_i$: 
$f(C)$ is also innermost in $\cup_{i=1}^\mu D_i$.
Let $\Delta$ be the disc bounded by $C$ in $\cup_{i=1}^\mu D_i$: it is 
entirely contained in one of the two balls bounded by $S$ that we shall denote 
by $B^+$. The disc $f(\Delta)$, bounded by $f(C)$ is contained in the second 
ball, denoted by $B^-$. Consider now $\gamma_1(\Delta)\subset B^-$ and 
$\gamma_1(f(\Delta))=f(\gamma_1(\Delta))\subset B^+$. Since $\OO_{\mu}$ is 
$\gamma_1$-invariant,
the interiors of these 
discs
are disjoint from $\OO_{\mu}$ and 
also from $L$. Up to small isotopy, we can assume that the two discs meet the 
family $\cup_{i=1}^\mu D_i$ transversally. 
If the interior of $\gamma_1(\Delta)$ is disjoint from $\cup_{i=1}^\mu D_i$
(and so is the interior of $f(\gamma_1(\Delta))$),
then we can use these two 
discs to remove the intersection $C$ and $f(C)$.
Otherwise, we 
perform $f$-equivariant surgery along the family $\cup_{i=1}^\mu D_i$ in order 
to eliminate all intersections in the interior of the two discs 
$\gamma_1(\Delta)$ and $f(\gamma_1(\Delta))$,
as follows.
Indeed, let $C'$ be an innermost 
circle of intersection in $\gamma_1(\Delta)$. $C'$ must be contained in and 
bound a subdisc of some disc $D_i$. One can now replace the subdisc of $D_i$ 
with a copy of the disc bounded by $C'$ in $\gamma_1(\Delta)$ to reduce the 
intersection. At the same time, one can replace the subdisc bounded by $f(C')$ 
in $f(D_i)$ with the subdisc bounded by $f(C')$ in 
$f(\gamma_1(\Delta))$: note that these operations take place in disjoint balls. 
We stress again that such surgery can only diminish the number of components of 
$S\cap(\cup_{i=1}^\mu \interior D_i)$
because $C'$ and $f(C')$ are not necessarily innermost
in $\cup_{i=1}^\mu D_i$. 

We continue to denote by 
$\{D_i\}_{i=1}^\mu$
the family obtained 
after surgery. 
Let $C$ be the circle chosen at the beginning of the preceding paragraph.
If it is no longer 
contained in 
$\cup_{i=1}^\mu D_i$
then there is nothing to do;
note that in this case the intersection $f(C)$ has also been removed. 
Suppose $C$ is contained in $\cup_{i=1}^\mu D_i$.
If the interior of $\gamma_1(\Delta)$ is not disjoint from $\cup_{i=1}^\mu D_i$,
then we repeat the preceding argument to decrease the intersection.
If the interior of $\gamma_1(\Delta)$ is disjoint from $\cup_{i=1}^\mu D_i$,
then can use $\gamma_1(\Delta)$ and $f(\gamma_1(\Delta))$ to remove $C$ and 
$f(C)$, as in the preceding paragraph.

The above shows that the family can be chosen so that 
$S\cap(\cup_{i=1}^\mu \interior D_i)$
does not contain circle components.
This implies immediately that no component of $\OO_{\mu}$ can be disjoint from
the sphere $S$ or contained in it for in this case the link $L$ would be
split,
contrary to the assumption that it is hyperbolic: indeed if $O_i$ is any
such component, the interior of the disc $D_i$ in the family just constructed
is disjoint from $\OO_{\mu}$ and $S$, and thus does not meet $K_0$ either.   
This completes the proof of 
the assertion (1) of the proposition.

(2) Suppose one component, say $O_1$, of $\OO_{\mu}$ is $f$-invariant.
By Theorem \ref{thm2}(1), the other components of $\OO_{\mu}$ 
are not $f$-invariant. Then by 
the equivariant Dehn's lemma,
there is an $f$-invariant family of $\mu+1$ discs 
$\{D_1',D_1''\}\cup\{D_i\}_{i=2}^\mu$ with disjoint interiors, such that 
$\partial D_1' =\partial D_1''=O_1$, $\partial D_i=O_i$ ($2\le i\le \mu$).
We can assume that this family intersects the sphere $S$ transversally.
For each loop component $C$ of 
$\mathcal{I}:=S\cap ((\interior D_1'\cup \interior D_1'')\cup(\cup_{i=2}^{\mu} \interior D_i))$,
we define its inside and its type as in the proof of (1).
Note that each of $S\cap D_1'$ and $S\cap D_1''$
contains a unique arc component, denoted by 
$\alpha_1'$ and $\alpha_1''$
respectively, and the union 
$\alpha_1'\cup \alpha_1''$
forms a great circle in 
$S$. This implies that no loop component of $\mathcal{I}$
contains 
$\alpha_1'$ or $\alpha_1''$
in its inside.
(It should be also noted that the loop $\alpha_1'\cup\alpha_1''$
is not contained in $\mathcal{I}$.)
Now, the argument in the proof of (1) works verbatim, 
and we can remove all loop components of $\mathcal{I}$,
completing the proof of (2) of the proposition.
\end{proof}

\begin{remark}\label{rem:arc-presentation}
{\rm
In the above proposition, the link $L=K_0\cup \OO_{\mu}$
is recovered from the $f$-invariant arc system in $S$
which is obtained as the intersection of the $f$-invariant family of disks with $S$.
To explain this, identify $\threesphere$ with the suspension of $S$,
the space obtained from $S\times [-1,1]$ by identifying the subspaces
$S\times \{\pm 1\}$ to a point. 
We assume that the $\G$-action on $\threesphere$ is equivalent to 
the $\G$-action on the suspension 
that is obtained from the natural product action of $\G$ on $S\times[-1,1]$.
Then the following hold.

\begin{enumerate}
\item
In the first case, set $\alpha_i=D_i\cap S$ ($1\le i\le \mu)$.
Then $L$ is $\G$-equivariantly homeomorphic to the link in the suspension
obtained as the image of
\[
K_0\cup(\cup_{i=1}^{\mu} \partial(\alpha_i\times [-1/2,1/2])\subset S\times [-1,1].
\]
\item
In the second case, set $\alpha_1'=D_1'\cap S$ and
$\alpha_i=D_i\cap S$ ($2\le i\le \mu)$.
Then $L$ is $\G$-equivariantly homeomorphic to the link in the suspension
obtained as the image of
\[
K_0\cup\partial(\alpha_1'\times [-1,1])\cup(\cup_{i=2}^{\mu} 
\partial(\alpha_i\times [-1/2,1/2]))
\subset S\times [-1,1].
\]
It should be noted that the image of $\partial(\alpha_1'\times [-1,1])$
in the suspension of $S$ is the suspension of $\partial \alpha_1'=O_1\cap S\subset S$.
Moreover, if $\alpha_1$ is any arc in $S$ with endpoints $O_1\cap S$
such that $\alpha_1\cap f(\alpha_1)=\partial \alpha_1$,
then $L$ is ($\langle \gamma_1\rangle$-equivariantly, but not $\G$-equivariantly)
homeomorphic to the link in the suspension
obtained as the image of
\[
K_0\cup\partial(\alpha_1\times [-2/3,2/3])\cup(\cup_{i=2}^{\mu} 
\partial(\alpha_i\times [-1/2,1/2]))
\subset S\times [-1,1].
\]
\end{enumerate}
For example, the links $L_6$ and $L_3$, respectively, satisfy the condition (1) 
and (2) of Proposition \ref{p:K0inS}, and they are represented by the arcs systems 
(1) and (2) in Figure~\ref{fig:arc-system}.
}
\end{remark}

\begin{figure}[h]
\begin{center}
 {
  \includegraphics[height=5cm]{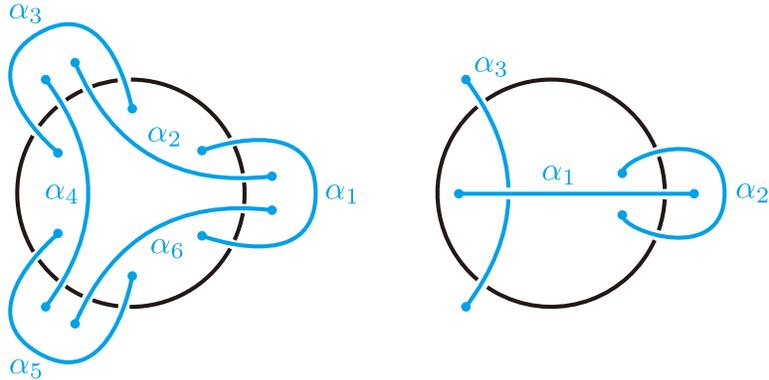}
 }
\end{center}
\caption{The arc systems for a $L_6$ and $L_3$.}
\label{fig:arc-system}
\end{figure}

Next, we present the following generalisation of 
Remark~\ref{r:strongly-invertible}(1).

\begin{Proposition}
\label{p:strongly-invertible}
Let $K$ be an invertible amphicheiral knot having free period $2$ and let
$L=K_0\cup \OO_\mu$ be its root. Let $\gamma_i$ and $\gamma_i'$ respectively, 
$i=1,2$, be the symmetries of $L$ generating the two $\G$-actions (compare 
Remark~\ref{rem:invertible} and Subsection~\ref{ss:l3}). 
Then any extension $\hat h$ 
of the element $h=\gamma_1\gamma_1'$ is never a strong inversion of $K$.  
\end{Proposition}

\begin{proof}
Assume by contradiction that $\hat h$ is a strong inversion of $K$ extending 
$h$. Recall that according to Theorem~\ref{thm2}, $K_0$ is a trivial knot 
contained in the $2$-sphere $S=\Fix(\gamma_1)$, moreover the $2$-sphere 
$S'=\Fix(\gamma'_1)$ intersects $S$ perpendicularly and meets $K_0$ in two 
antipodal points. Consider now the sublink $\OO=\OO_\mu$ of $L$: 
according to Theorem~\ref{thm2} it must contain two components that intersect 
transversally $S$ and two components that intersect transversally $S'$. We note 
that, under our hypotheses, a component that meets $S$ (respectively $S'$) 
transversally cannot intersect $S'$ (respectively $S$) transversally.
This follows from the fact that, by the consideration in 
Remark~\ref{r:strongly-invertible}, $h$ cannot leave a component invariant and 
act as a rotation on it, so that a component either intersects $\Fix(h)$ or its
linking number with $\Fix(h)=S\cap S'$ must be even, and hence zero. 
This implies that $\OO$ contains (at least two) components that intersect $S'$ 
transversally and are either contained in $S$ or disjoint from it. 
This is however impossible according to Proposition~\ref{p:K0inS}. 
\end{proof}

Finally, we show that the first assertion of Theorem~\ref{thm2}
(i.e., Claim~\ref{c:inv-comp}) is \lq\lq best possible",
in the sense that all situations described in the claim can arise.

Let $L=K_0\cup\OO_{\mu}$ be a link which provides an admissible root.
We first assume that $K_0$ is the unique $\G$-invariant component of $L$.
Then $L$ satisfies one of the following 
conditions.

\begin{enumerate}
\item
$K_0$ is a trivial knot contained in $Fix(\gamma_1)$.
\item
$K_0$ is a trivial knot meeting $Fix(\gamma_1)$ in two points.
\item
$K_0$ is a composite knot meeting $Fix(\gamma_1)$ in two points.
\end{enumerate}
The links $L_6$ and $L_2$ in Section~\ref{section:roots} provide examples of 
the first and the third situations, respectively.
We show that the second situation can also occur.
Consider the configuration of Figure~\ref{fig:trivialcomp}, where each small 
box represents a rational tangle such that (i) the strands of the tangles 
behave combinatorially as the dotted arcs inside the boxes, and (ii) the 
tangles are not a sequence of twists.

\begin{figure}[h]
\begin{center}
 {
  \includegraphics[height=5cm]{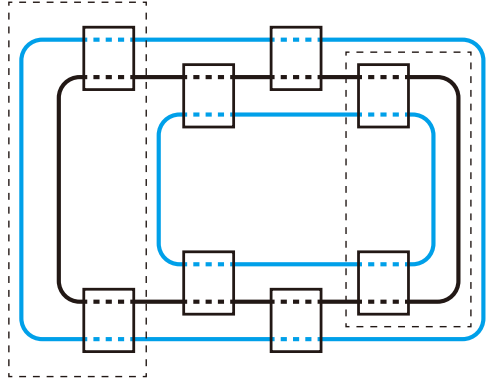}
 }
\end{center}
\caption{A link providing an admissible root with a single trivial
$\G$-invariant component containing $Fix(\gamma_2)$,
which forms a trivial knot.}
\label{fig:trivialcomp}
\end{figure}

For each choice of rational tangles as above, the result is a link with three 
trivial components (note that when forgetting 
the outer- respectively inner-component, 
the inner- respectively outer-component 
and the central one form a 
Montesinos link). The exterior of the the four central small boxes
and of the dotted ones is a \emph{basic polyhedron} (see \cite[ch. 10]{Ka}), so 
it is $\pi$-hyperbolic by Andreev's theorem. The interiors of the dotted boxes 
have Seifert fibred double covers, regardless of the chosen rational tangles. 
It follows that, for sufficiently large rational
tangles inserted into the four central 
small boxes, the exterior of the dotted boxes is $\pi$-hyperbolic. It follows 
that, for sufficiently large 
rational
tangles, the Bonahon-Siebenmann decomposition 
\cite{BS1} of the orbifold associated to the link with branching order $2$ 
consists of three geometric pieces: two Seifert fibred ones, and a 
$\pi$-hyperbolic one which is their complement. Since the Seifert fibred ones 
are atoroidal (and the $\pi$-hyperbolic one is anannular), the link is 
hyperbolic. It is now easy to see that the construction can be carried out in
a $\G$-equivariant way so that the black central component is $\G$-invariant,
providing an admissible root with the desired property.

\medskip

We next consider the case where $L$ has two $\G$-invariant components.
Then these can be both trivial, as is the case of $L_3$, or one trivial and one
composite. An example of the latter type can be built from the same tangle 
in the half ball
used to construct $L_2$ but with a different gluing and an extra component 
contained in $Fix(\gamma_1)$, see Figure~\ref{fig:mixcomp}. Hyperbolicity of
this link can be checked using   
a computer program like SnapPea, SnapPy, or HIKMOT,
or can be proved following the same lines as the proof provided for $L_2$ (see Section~\ref{section:link2}) by a slight 
adaptation of Lemma~\ref{lem:simple-1-4}, since $K_D$ is in a different 
position: details are left to the interested reader.

\begin{figure}[h]
\begin{center}
 {
  \includegraphics[height=5cm]{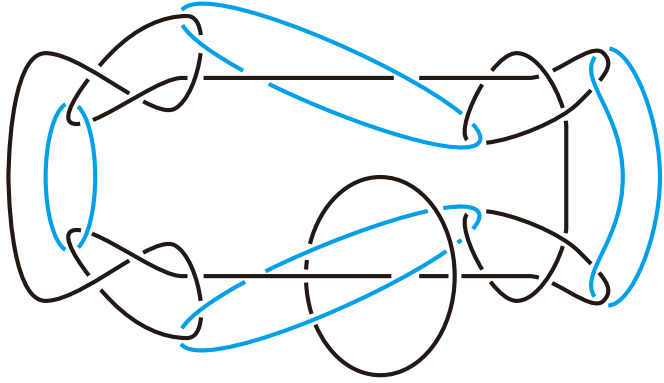}
 }
\end{center}
\caption{A link providing an admissible root with two $\G$-invariant components
consisting of a trivial knot and a composite one.}
\label{fig:mixcomp}
\end{figure}

These two observations say that 
the first assertion of Theorem~\ref{thm2}
(i.e., Claim \ref{c:inv-comp}) is \lq\lq best possible".


\section{The link $L_2$ is hyperbolic}
\label{section:link2}


Consider the link $L_2=K_0\cup \OO_2$ with $\OO_2=O_1\cup O_2$ in
Figure~\ref{fig:root}(3). Then $L_2$ is the double of the tangle
$(\threeball, \tau):=(\threeball, \tau_0\cup \tau_1\cup \tau_2)$ in
Figure \ref{fig:root}(2).
The tangle $(\threeball, \tau)$ is regarded as the sum of the tangle
$(\threeball, t):=(\threeball, t_0\cup t_1\cup t_2)$ in
Figure~\ref{fig:root}(1) with its mirror image.

\begin{Lemma}
\label{lem:simple-1-1}
The tangle $(\threeball, t_0\cup t_1)$ obtained by
removing one of the unknotted arcs, as in Figure~\ref{fig:tangle}, is
hyperbolic with totally geodesic boundary.
To be precise, 
$B^3\setminus(t_0\cup t_1)$
admits a complete hyperbolic structure,
such that $\partial(B^3\setminus(t_0\cup t_1))$
is a totally geodesic surface.
\end{Lemma}

\begin{proof}
If we take the double of the tangle,
we obtain the
pretzel link $P(3,2,-2,-3)$,
which is hyperbolic by \cite{BS} (see also \cite{BZ, FG}).
The desired result follows from this fact.
\end{proof}

\begin{figure}[h]
\begin{center}
 {
  \includegraphics[height=5cm]{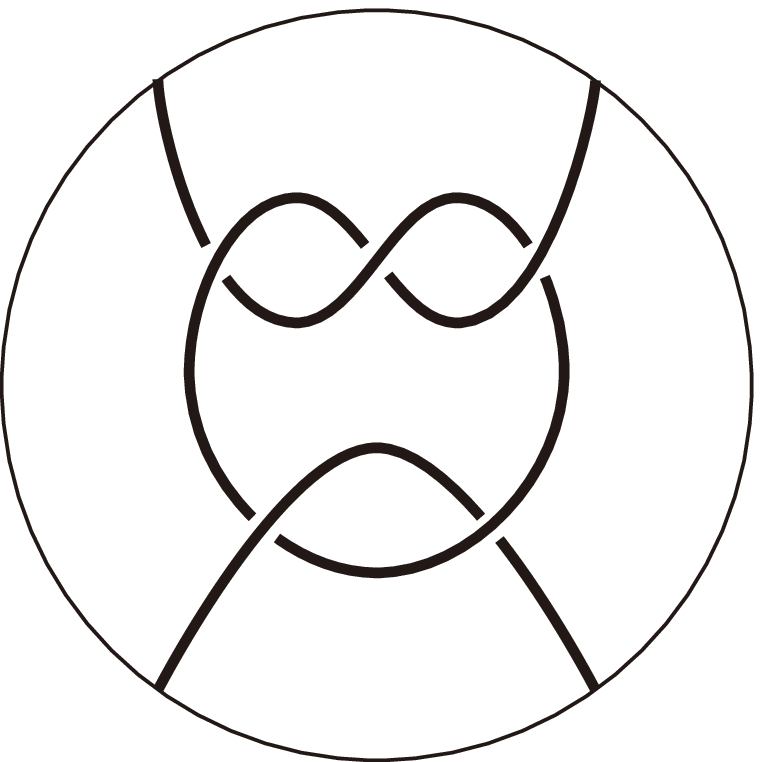}
 }
\end{center}
\caption{The tangle $(\threeball, t_0\cup t_1)$.}
\label{fig:tangle}
\end{figure}

\begin{Lemma}
\label{lem:simple-1-2}
The link, $L'$, obtained as the double of $(\threeball, t_0\cup t_1\cup t_2)$,
is prime and unsplittable.
\end{Lemma}

\begin{proof}
Observe that the link $L'$ is obtained from the link $L'':=P(3,2,-2,-3)$ by 
adding a parallel copy of the unknotted component.
Thus the exterior $E(L')$ is obtained
from the hyperbolic manifold $E(L'')$ and the Seifert fibered space $P\times 
{\mathbf S}^1$ with $P$ a two-holed disc, by gluing along 
two incompressible toral boundary components. 
Since both $E(L'')$ and $P\times {\mathbf S}^1$ are irreducible
and since the torus 
$\partial P\times {\mathbf S}^1$
is incompressible both in $E(L'')$ and $P\times {\mathbf 
S}^1$, we see that $E(L')$ is irreducible. 
Hence 
$L'$
is unsplittable. Since 
$E(L'')$ admits no essential annuli, we see that the only essential annuli in 
$E(L')$ are saturated annuli contained in $P\times {\mathbf S}^1$ with boundary
components contained in $\partial (P\times {\mathbf S}^1)\setminus \partial 
E(L'')$. Hence 
$L'$
is prime.
\end{proof}

\begin{Lemma}
\label{lem:simple-1-3}
{\rm (1)} Let $\Delta$ be a disc properly embedded in $\threeball$ such that
$\Delta$ is disjoint from $t$. Then $\Delta$ cuts off a $3$-ball in
$\threeball$ disjoint from $t$.

{\rm (2)} Let $\Delta$ be a disc properly embedded in $\threeball$ such that
$\Delta$ intersects $t$ transversely in a single point. Then $\Delta$ cuts off a
$3$-ball, $B$, in $\threeball$ such that $(B, t\cap B)$ is a $1$-string
trivial tangle.
\end{Lemma}

\begin{proof}
Suppose that there exists a disc, $\Delta$, properly embedded in $\threeball$
such that either $\Delta$ is disjoint from $t$ or $\Delta$ intersects $t$
transversely in a single point, and which does not satisfy the desired
conditions. Then the double of $\Delta$ gives a $2$-sphere in $\threesphere$
which separates the components of $L'$ or gives a nontrivial decomposition of
$L'$. This contradicts Lemma \ref{lem:simple-1-2}.
\end{proof}

Let $D$ be the flat disc forming the right side of $\partial \threeball$ as illustrated in
Figure~\ref{fig:root}(1), and set $K_D=\partial D$.

\begin{Lemma}
\label{lem:simple-1-4}
There does not exist an essential annulus, $A$, in $\threeball\setminus t$ whose
boundary is disjoint from $K_D$. To be precise, any incompressible
annuls $A$ properly embedded in $\threeball\setminus t$ such that
$\partial A\cap K_D=\emptyset$ is parallel in $\threeball\setminus t$ to
(i) an annulus in $\partial\threeball\setminus (K_D\cup t)$,
(ii) the frontier of a regular neighbourhood of $K_D$ or
(iii) the frontier of a regular neighbourhood of a component $t_i$ of $t$.
\end{Lemma}

\begin{proof}
Suppose to the contrary that there is an essential annulus, $A$, in
$\threeball\setminus t$ such that $\partial A\cap K_D=\emptyset$. Since $A$ is
compressible in $\threeball$, it bounds a cylinder $B^2\times I$ with
$I=[0,1]$,
where $A=\partial B^2\times I$ and $B^2\times \partial I\subset
\partial\threeball$.

\medskip

{\bf Case 1.} One component of $\partial A$ is contained in $\interior D$ and
the other component is contained in $D^c:=\partial \threeball \setminus D$. Then we may
assume $B^2\times 0\subset \interior D$ and $B^2\times 1\subset D^c$. If $A$ is
compressible in $\threeball \setminus (t_0\cup t_1)$, then the cylinder must contain
$t_2$. Since $t_2$ has an endpoint in $\interior D$ and the other endpoint in
$D^c$, $t_2$ has an endpoint in $B^2\times 0$ and the other endpoint in
$B^2\times 1$. Since $t_2$ is unknotted, this implies $t_2$ forms the core of
the cylinder and hence $A$ is parallel to the frontier of a regular
neighbourhood of $t_2$. This contradicts the assumption that $A$ is essential in
$\threeball\setminus t$. So, we may assume $A$ is incompressible in
$\threeball \setminus (t_0\cup t_1)$. 
Thus Lemma~\ref{lem:simple-1-1} implies that $A$ is parallel in $\threeball\setminus (t_0\cup t_1)$
to (i) the frontier of a regular neighbourhood of $t_0$, (ii) the frontier of
a regular neighbourhood of $t_1$ or (iii) an annulus in
$\partial \threeball \setminus (t_0\cup t_1)$.

\medskip
{\bf Subcase 1-i.}
$A$ is parallel in $\threeball\setminus (t_0\cup t_1)$ to the frontier of a regular
neighbourhood $N(t_0)$ of $t_0$. Since $t_0$ is knotted whereas $t_1$ and $t_2$
are unknotted and since each $t_i$ has an endpoint in $\interior D$ and the
other endpoint in $D^c$, we see that $t_1$ and $t_2$ are not contained in
$N(t_0)$. Hence $A$ is parallel in $\threeball\setminus t$ to the frontier of $N(t_0)$,
a contradiction to the assumption that $A$ is essential.

\begin{figure}[h]
\begin{center}
 {
  \includegraphics[height=5cm]{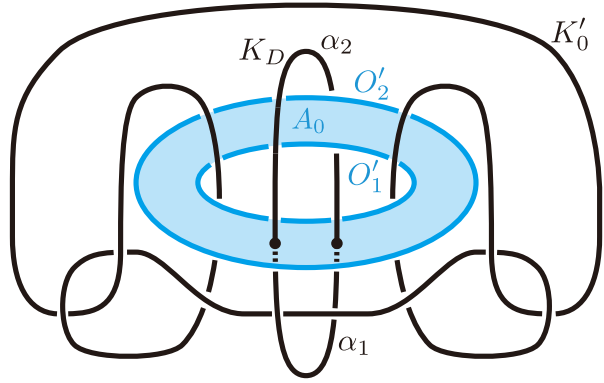}
 }
\end{center}
\caption{The link $L'\cup K_D=(K_0'\cup O_1'\cup O_2')\cup K_D$ 
obtained as the ``double'' of $(\threeball, t\cup K_D)$,
where the components $K_0'$, $O_1'$ and $O_2'$ are the doubles of $t_0$, $t_1$ and $t_2$, respectively,
and the component $K_D$ is a copy of the loop $K_D\subset \partial \threeball$.
The component $K_D$ intersects the annulus $A_0$ bounded by $O'_1$ and $O'_2$ transversely in two points,
which divides $K_D$ into the two arcs $\alpha_1$ and $\alpha_2$.}
\label{fig:L3-KD}
\end{figure}

{\bf Subcase 1-ii.}
$A$ is parallel in $\threeball\setminus (t_0\cup t_1)$ to the frontier of a regular
neighbourhood $N(t_1)$ of $t_1$ in $\threeball\setminus t_0$. If $t_2$ is not contained
in $N(t_1)$, then $A$ is parallel in $\threeball\setminus t$ to the frontier of
$N(t_1)$. So $t_2$ must be contained in $N(t_1)$. Now, let $L'\cup K_D$ be the
link in $\threesphere$ obtained as the ``double'' of $(\threeball, t\cup K_D)$,
and let $T$ be the torus in the exterior $E(L'\cup K_D)$ obtained as the double
of the annulus $A$. Then, by the above observation, $T$ bounds a solid torus,
$V$, in $\threesphere$, such that $V\cap (L'\cup K_D)= O_1'\cup O_2'$, where
$O_i'$ is the double of $t_i$ ($i=1,2$). By the proof of
Lemma~\ref{lem:simple-1-2}, we know that the JSJ decomposition of $E(L')$
is given by the torus $T_0:=\partial E(L'')=\partial (P\times {\mathbf S}^1)$,
and hence $T$ is isotopic to $T_0$ in $E(L')$. Thus there is a
self-diffeomorphism $\psi$ of $(\threesphere, L')$
pairwise isotopic to the identity, which carries $T_0$ to $T$. 
Let $A_0$ be the annulus with $\partial A_0=O_1'\cup O_2'$ as illustrated in
Figure~\ref{fig:L3-KD}. Note that, up to isotopy, $T_0$ is the boundary of a
regular neighbourhood of $A_0$
Note that the component $K_D$
intersects $A_0$ transversely in two points, whereas $K_D$ is disjoint from
$\psi(T_0)=T\subset E(L'\cup K_D)$ and therefore $K_D$ is disjoint from $\psi(A_0)$.
Since $\psi$ is pairwise isotopic to the identity, there is a smooth isotopy
$\Psi:K_D\times [0,1]\to E(L')$ such that $\Psi|_{K_D\times 0}=1_{K_D}$ and
$\Psi(K_D\times 1)=\psi^{-1}(K_D)$. We may assume $\Psi$ is transversal to $A_0$
and so $F^{-1}(A_0)$ is a $1$-dimensional submanifold of  $K_D\times [0,1]$.
Since $K_D\cap A_0$ consists of two transversal intersection points and
since $\psi^{-1}(K_D)\cap A_0=\emptyset$,
the $1$-manifold $\Psi^{-1}(A_0)$ contains precisely one arc component, $\beta$, 
and it joins the two points $(K_D\cap A_0)\times 0$.
Consider the disc, $\delta$, in $K_D\times [0,1]$ cut off by the arc $\beta$.
Since $A_0\cap E(L')$ is incompressible in $E(L')$ and since
$E(L')$ is irreducible, we may assume that
the interior of $\delta$ is disjoint from $\Psi^{-1}(A_0)$. 
Thus the loop $\Psi(\partial \delta)$ is null-homotopic in $E(L')$.
On the other hand, the loop $\Psi(\partial \delta)$ is the union of one of the two subarcs 
$\alpha_1$ and $\alpha_2$ of $K_D$ bounded by $K_D\cap A_0=\partial \beta$
(see Figure~\ref{fig:L3-KD})
and the path $\Psi(\beta)$ in $A_0$
joining the two points $K_D\cap A_0$.
However, we can easily observe that 
$\mathrm{lk}(\alpha_1\cup \beta',K_0')=\pm1$ and
$\mathrm{lk}(\alpha_2\cup \beta',O_2')=\pm1$ for any path $\beta'$ in
$A_0$ with endpoints $K_D\cap A_0$. This is a contradiction.

\medskip

{\bf Subcase 1-iii.}
$A$ is parallel in $\threeball\setminus (t_0\cup t_1)$ to an annulus, $A'$, in
$\partial \threeball \setminus (t_0\cup t_1)$. Then $A$ cuts off a solid torus, $W$,
from $\threeball \setminus (t_0\cup t_1)$. Since $A$ is essential, $W$ must contain the
component $t_2$. This implies that $O_2'$ is null homotopic in
$\threesphere \setminus K_0'$, where $K_0'$ is the double of $t_0$. However, we can
easily check by studying the knot group of the square knot $K_0'$, that this is
not the case. 

\medskip
{\bf Case 2.} Both components of $\partial A$ are contained in $\interior D$.
Since each component of $t_i$ has one endpoint in $\interior D$ and the other
endpoint in $D^c$, we see that the components of $\partial A$ are concentric in
$D$, and that $B^2\times 0$ is contained in $\interior D$ and that
$B^2\times 1$ contains $D^c$ in its interior. Since $t_1$ joins $\interior D$
with $D^c$, $t_1$ joins  $B^2\times 0$ with $B^2\times 1$ in the cylinder
$B^2\times I$. So the cylinder $B^2\times I$ is unknotted in $\threeball$ and
hence the closure of its complement in $\threeball$ is a solid torus, $W$.
Since $\partial B^2\times 0$ is a meridian of the solid torus
$(\threesphere\setminus \interior\threeball)\cup B^2\times I$, we see that
$\partial B^2\times 0$ is a longitude of $W$. So $A$ is parallel to the annulus
$A':=V\cap \partial \threeball$ through $W$. Since every component of $t$ joins
$\interior D$ with $D^c$, $t$ must be contained in the cylinder and hence $W$
is disjoint from $t$. This contradicts the assumption that $A$ is essential.
\end{proof}

\begin{Lemma}
\label{lem:simple-1-5}
The complement $\threeball\setminus t$ does not contain an incompressible torus.
\end{Lemma}

\begin{proof}
This follows from the easily observed fact that the exterior $E(t)$ is
homeomorphic to a genus $3$ handlebody. Indeed, the exterior of $t$ in
$\threeball$ is the trefoil knot exterior with two parallel unkotting tunnels
drilled out.
\end{proof}

\begin{Lemma}
\label{lem:simple-1-6}
The tangle $(\threeball,\tau)$ is simple in the following sense.
\begin{enumerate}
\item Let $\Delta$ be a disc properly embedded in $\threeball$ such that
$\Delta$ is disjoint from $\tau$. Then $\Delta$ cuts off a $3$-ball in
$\threeball$ disjoint from $\tau$.
\item Let $\Delta$ be a disc properly embedded in $\threeball$ such that
$\Delta$ intersects $\tau$ transversely in a single point. Then $\Delta$ cuts
off a $3$-ball, $B$, in $\threeball$ such that $(B, \tau\cap B)$ is a
$1$-string trivial tangle.
\item There does not exist an essential annulus, $A$, in $\threeball\setminus \tau$.
To be precise, any incompressible annuls $A$ properly embedded in
$\threeball\setminus \tau$ is parallel in $\threeball\setminus \tau$ to
(i) an annulus in $\partial\threeball\setminus \tau$,
(ii) the frontier of a regular neighbourhood of a component $\tau_i$ of $\tau$.
\item $\threeball\setminus \tau$ does not contain an incompressible torus.
\end{enumerate}
\end{Lemma}

\begin{proof}
Note that $(\threeball,\tau)$ is obtained from $(\threeball,t)$ and its mirror
image by identifying $D$ with its copy in the mirror image. Thus the assertion
follows from Lemmas~\ref{lem:simple-1-2}-\ref{lem:simple-1-5} by using the
standard cut and paste method (see \cite[Criterion 6.1]{S1}).
\end{proof}

Since $(\threesphere,L_2)$ is the double of the simple tangle
$(\threeball,\tau)$, 
we see that $\threesphere\setminus L_2$ is atoroidal, i.e, it does not contain 
an essential torus.
Moreover, $\threesphere\setminus L_2$ is not a Seifert fibred space.
This can be seen as follows.
If it were a Seifert fibred space, then
it should be a $2$-fold composing space, 
i.e., homeomorphic to the complement of the connected sum of two Hopf links.
Now we use the fact that for a given link $L$,
the greatest common divisor, $d(L)$, 
of the linking numbers of the components is an
invariant of the link complement.
Since $d(\mbox{the connected sum of two Hopf links})=1$ and $d(L_2)=0$,
$\threesphere\setminus L_2$ cannot be a Seifert fibred space.
Hence, by Thurston's uniformization theorem for Haken manifolds,
$\threesphere \setminus  L_2$ admits a complete hyperbolic structure,
i.e., $L_2$ is hyperbolic.


\section{The link $L_3$ is hyperbolic}
\label{section:link3}


In this section, we prove that the link $L_3$ is hyperbolic. We begin with the
following lemma.

\begin{Lemma}
\label{lem:no-separating-two-sphere}
There does not exist a $2$-sphere, $S$, in $\threesphere$ satisfying the
following conditions.
\begin{enumerate}
\item $S$ is disjoint from $K_0$.
\item Either $S$ is disjoint from $O_1$ or $S$ intersects $O_1$ transversely in
two points.
\item $S$ separates $O_2$ from $O_3$, i.e., $O_2$ and $O_3$ are contained in
distinct components of $\threesphere\setminus S$.
\end{enumerate}
\end{Lemma}

\begin{proof}
Suppose to the contrary that there is a $2$-sphere $S$ satisfying the
conditions. Let $\branchedcover$ be the $\Z/2\Z\oplus\Z/2\Z$-covering of
$\threesphere$ branched over the Hopf link $K_0\cup O_1$, and let $\tilde S$
and $\tilde O_i$, respectively, be the inverse image of $S$ and $O_i$ in
$\branchedcover$. Then $\tilde O_2\cup \tilde O_3$ is the link in the
$3$-sphere $\branchedcover$ as illustrated in Figure~\ref{fig:Branched-Cover}.
Observe that any component of $\tilde O_2$ and any component of $\tilde O_3$
form a Hopf link. On the other hand, the inverse image of $S$ in
$\branchedcover$ consists of four or two mutually disjoint $2$-spheres, each of
which separates a component of $\tilde O_2$ from each component of 
$\tilde O_3$. This is a contradiction.
\end{proof}

\begin{figure}[h]
\begin{center}
 {
  \includegraphics[height=10cm]{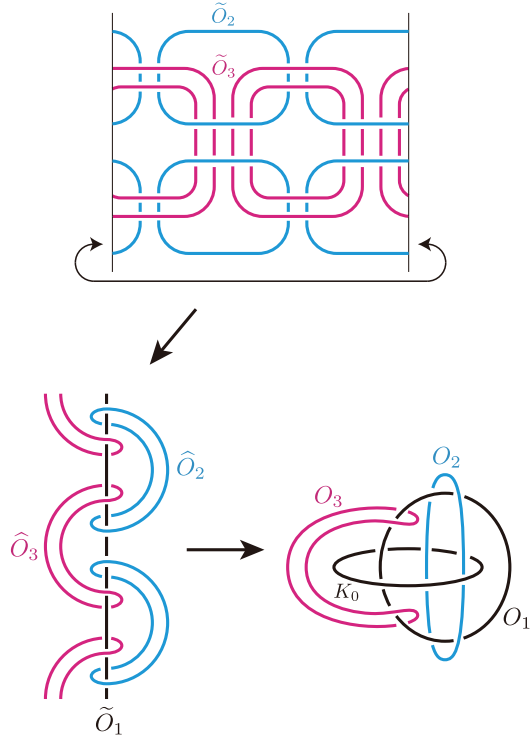}
 }
\end{center}
\caption{The bottom left drawing illustrates 
the double cover, $M_2(K_0)$, of $\threesphere$
branched along $K_0$ and the inverse images of $\OO_3=O_1\cup O_2\cup O_3$.
The $\Z/2\Z\times \Z/2\Z$-cover, $M_{2\oplus 2}(K_0\cup O_1)$, of $\threesphere$ branched along the 
Hopf link $K_0\cup O_1$ is obtained as the double cover of $M_2(K_0)$
branched along the inverse image of $O_1$.
The drawing on the top illustrates the inverse image of $O_2\cup O_3$
in $M_{2\oplus 2}(K_0\cup O_1)$,
where the left and right sides are glued together.}
\label{fig:Branched-Cover}
\end{figure}

Let $\threeball$ be one of the $3$-balls in $\threesphere$ bounded by
$\twosphere$, and set $t_i:=O_i\cap \threeball$ ($i=1,2,3$). Then the pair
$(\threeball,t)$ with $t=t_1\cup t_2\cup t_3$ is a $3$-string tangle, and $K_0$
is a loop in $\partial\threeball$. Let $D$ and $D'$ be the $2$-discs in
$\partial \threeball$ bounded by $K_0$ which contains $\partial t_2$ and
$\partial t_3$, respectively.

\begin{Lemma}
\label{lem:no-compressing-disk}
The three-punctured disc $D\setminus t$ is incompressible in 
$\threeball\setminus t$. Namely, there does not exist a disc, $\Delta$, 
properly embedded in $\threeball$ satisfying the following conditions.
\begin{enumerate}
\item $\partial \Delta$ is a loop in $\interior(D\setminus t)$ which does not 
bound a disc in $\interior(D\setminus t)$.
\item $\Delta$ is disjoint from $t$.
\end{enumerate}
\end{Lemma}

\begin{proof}
Suppose to the contrary that there exists a disc $\Delta$ satisfying the
conditions. Let $D_{\Delta}$ be the disc in $D$ bounded by $\partial\Delta$,
and let $B$ be the $3$-ball in $\threeball$ bounded by the $2$-sphere
$\Delta\cup D_{\Delta}$. Since each of $t_1$ and $t_3$ has an endpoint in
$\interior D'\subset \partial\threeball\setminus D_{\Delta}$ and since it is disjoint
from $\Delta$, both $t_1$ and $t_3$ are contained in the complement of $B$.
So, by the second condition, $t_2$ has an endpoint in $\interior D_{\Delta}$,
and therefore $t_2$ is contained in $B$. Hence, by taking the double of
$\Delta$, we obtain a $2$-sphere in $\threesphere$ satisfying the conditions in
Lemma~\ref{lem:no-separating-two-sphere}, which is a contradiction.
\end{proof}

\begin{Lemma}
\label{lem:indivisible-tangle}
There does not exist a disc, $\Delta$, properly embedded in $\threeball$
satisfying the following conditions.
\begin{enumerate}
\item $\partial \Delta$ is a loop in $\interior(D\setminus t)$.
\item $\Delta$ intersects $t$ transversely in a single point.
\item $\Delta$ does not cut off a $1$-string trivial tangle from
$(\threeball,t)$.
\end{enumerate}
\end{Lemma}

\begin{proof}
Suppose to the contrary that there exists a disc $\Delta$ satisfying the
conditions. As in the proof of Lemma \ref{lem:no-compressing-disk}, let
$D_{\Delta}$ be the disc in $D$ bounded by $\partial\Delta$, and let $B$ be the
$3$-ball in $\threeball$ bounded by $\Delta\cup D_{\Delta}$. Since the
endpoints of $t_3$ are contained in
$\interior D'\subset \partial\threeball\setminus D_{\Delta}$ and since $t_3$ intersects
$\Delta$ in at most one point, we see that $t_3$ is contained in the
complement of $B$. So, precisely one of  $t_1$ or $t_2$ intersects $\Delta$.
Suppose that $\Delta$ intersects $t_1$. Then, since $t_1$ is an unknotted arc,
the third condition implies that the $3$-ball $B$ must contain $t_2$.
Hence, by taking the double of $\Delta$, we obtain a $2$-sphere in
$\threesphere$ satisfying the conditions in
Lemma~\ref{lem:no-separating-two-sphere}, which is a contradiction.
Hence $t_1$ is disjoint from $\Delta$ and $t_2$ intersect $\Delta$ transversely
in a single point. Since both $t_1$ and $t_3$ are disjoint from $B$ and since
$t_2$ is an unknotted arc, we see that $(B,t\cap B)=(B,t_2\cap B)$ is a
trivial $1$-string tangle. This contradicts the assumption that $\Delta$
satisfies the third condition.
\end{proof}

Let $\delta_2$ and $\delta_3$ be the discs in $\threeball$ as illustrated in
Figure~\ref{fig:two-disks}.

\begin{figure}[h]
\begin{center}
 {
\includegraphics[height=5cm]{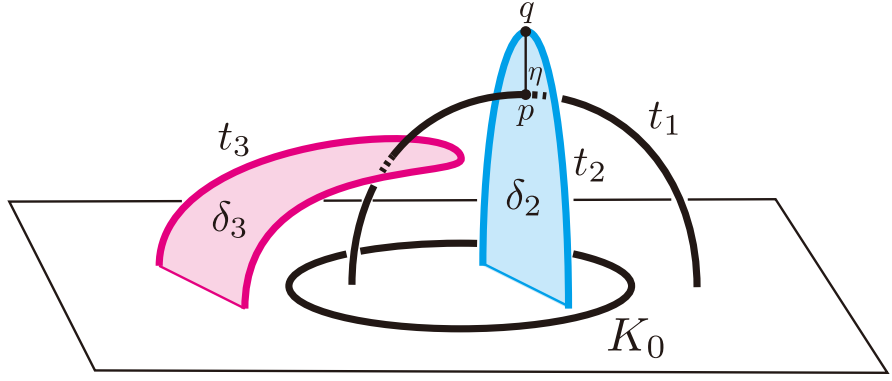}
 }
\end{center}
\caption{The discs $\delta_2$ and $\delta_3$ and the arc $\eta$.}
\label{fig:two-disks}
\end{figure}

\begin{Lemma}
\label{lem:no-boudary-compressing-disk}
There does not exist a disc, $\Delta$, in $\threeball$ satisfying the following
conditions.
\begin{enumerate}
\item $\partial \Delta=\alpha\cup\beta$, where $\alpha$ is an arc properly
embedded in $\delta_2\setminus t$ and $\beta$ is an arc in $\interior(D\setminus t)$ such that
$\beta\cap \delta_2=\partial\beta=\partial\alpha$.
\item $\Delta$ is disjoint from $t$.
\item The disc $\Delta$ is nontrivial in the following sense. 
Let
$\delta_{2,\Delta}$ be the subdisc of $\delta_2$ bounded by
$\alpha\cup \alpha'$, where $\alpha'$ is the subarc of $\delta_2\cap D$ bounded
by $\partial\alpha$. Then the disc $\Delta\cup \delta_{2,\Delta}$ does not cut
off a $3$-ball in $\threeball$ which is disjoint from $t$.
\end{enumerate}
\end{Lemma}

\begin{proof}
Suppose to the contrary that there exists a disc $\Delta$ satisfying the
conditions.
Let $D_{\Delta}$ be the subdisc of $D$ bounded by $\alpha'\cup \beta$.
Since $\Delta$ is disjoint from $t$, the loop
$\partial \delta_{2,\Delta}$ is homotopic to the loop $\partial D_{\Delta}$ in
$\threeball\setminus t$. Thus $\delta_{2,\Delta}$ contains the point $t_1\cap\delta_2$
if and only if $D_{\Delta}$ contains the point $t_1\cap D$.

Suppose first that $D_{\Delta}$ does not contain the point $t_1\cap D$
and so $\delta_{2,\Delta}$ does not contain the point $t_1\cap\delta_2$.
Then $\Delta\cup \delta_{2,\Delta}$ is a disc properly embedded in $\threeball$
disjoint from $t$. Hence the disc $D_{\Delta}$ in $D$ bounded by
$\partial(\Delta\cup \delta_{2,\Delta})$ is disjoint from $t$ by
Lemma~\ref{lem:no-compressing-disk}. Since $(\threeball,t)$ is a $3$-strand
trivial tangle, $\threeball \setminus  t$ is homeomorphic to a genus $3$ handlebody
(with three annuli on the boundary removed), and so $\threeball \setminus  t$ is
irreducible. Hence the $2$-sphere $\Delta\cup  \delta_{2,\Delta}\cup
D_{\Delta}$
bounds a $3$-ball in $\threeball \setminus  t$. 
This contradicts the assumption that
$\Delta$ satisfies the third condition that $\Delta$ is nontrivial.

Suppose next that $D_{\Delta}$ contains the point $t_1\cap D$ and so
$\delta_{2,\Delta}$ contains the point $t_1\cap\delta_2$. Consider an arc in
the disc $\delta_{2,\Delta}\cup D_{\Delta}$ joining the two points
$t_1\cap(\delta_{2,\Delta}\cup D_{\Delta})$, and let $\gamma$ be the boundary
of its regular neighbourhood in $\delta_{2,\Delta}\cup D_{\Delta}$. 
Then
$\partial\Delta=\partial(\delta_{2,\Delta}\cup D_{\Delta})$ is homotopic to
$\gamma$ in $\threeball \setminus  t$. 
Note that $\gamma$ is as illustrated in Figure \ref{fig:freegroup}.
If we ignore the string $t_2$ and regard $\gamma$ as a loop in 
$\threeball\setminus (t_1\cup t_3)$,
then the free homotopy class of $\gamma$
up to orientation
corresponds to 
the conjugacy class of the free group $\pi_1(\threeball\setminus (t_1\cup t_3))$
represented by $x_1(x_3x_1^{-1}x_3^{-1})$,
where $\{x_1, x_3\}$ is the free basis illustrated in 
Figure~\ref{fig:freegroup}. 
This contradicts the fact that $\Delta$ is disjoint from $t$.
This completes the proof of Lemma \ref{lem:no-boudary-compressing-disk}.
\end{proof}

\begin{figure}[h]
\begin{center}
 {
\includegraphics[height=5cm]{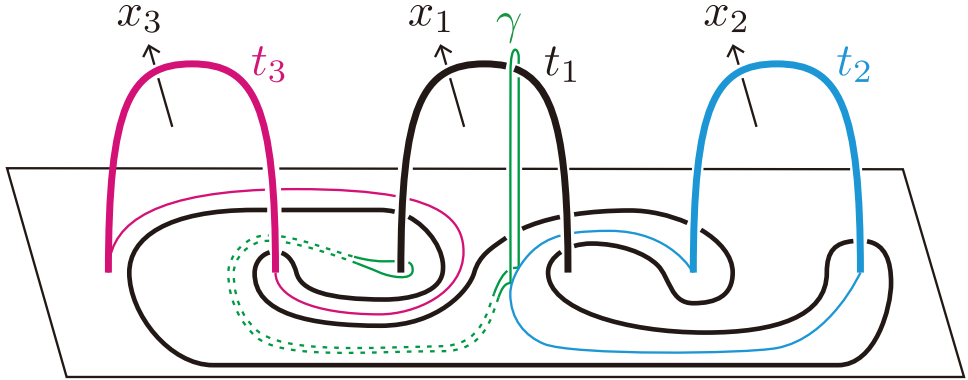}
 }
\end{center}
\caption{
The loop $\gamma$ in green and the free basis of the fundamental group of the
handlebody.}
\label{fig:freegroup}
\end{figure}

\begin{remark}
\label{rem:D-and-D'}
The restriction of the antipodal map $\gamma_2$ to $\threeball$ gives a
diffeomorphism of $(\threeball,K_0\cup t)$ interchanging $D$ with $D'$ (and
$\delta_2$ with $\delta_3$). Hence, we may replace $D$ with $D'$ (and
$\delta_2$ with $\delta_3$) in Lemmas~\ref{lem:no-compressing-disk},
\ref{lem:indivisible-tangle} and \ref{lem:no-boudary-compressing-disk}.
\end{remark}

\begin{Lemma}
\label{lem:no-essentail-annulus}
There does not exist an essential annulus, $A$, in $\threeball\setminus t$ whose
boundary is disjoint from $K_0=\partial D=\partial D'$. To be precise, any
incompressible annulus $A$ properly embedded in $\threeball\setminus t$ such that
$\partial A\cap K_0=\emptyset$ is parallel in $\threeball\setminus t$ to
(i) an annulus in $\partial\threeball\setminus (K_0\cup t)$,
(ii) the frontier of a regular neighbourhood of $K_0$ or
(iii) the frontier of a regular neighbourhood of a component $t_i$ of $t$.
\end{Lemma}

\begin{proof}
Suppose to the contrary that there is an essential annulus, $A$, in
$\threeball\setminus t$ such that $\partial A\cap K_0=\emptyset$. Since $A$ is
compressible in $\threeball$, it bounds a cylinder $B^2\times I$ with
$I=[0,1]$, where $A=\partial B^2\times I$ and $B^2\times \partial I\subset 
\partial\threeball$.

\medskip

{\bf Case 1.} One component of $\partial A$ is contained in $\interior D$
and the other component is contained in $\interior D'$. Then we may assume
$B^2\times 0\subset \interior D$ and $B^2\times 1\subset D'$. 
Since $A$ is
incompressible, $B^2\times 1$ must contain at least one point of $\partial t$.
If $B^2\times 1$ contains exactly one point of $\partial t$, then pushing
$A\cup (B^2\times 1)$ into the interior of $\threeball$, we obtain a properly
embedded disc $\Delta$ in $\threeball$ which satisfies the first two conditions
of Lemma \ref{lem:indivisible-tangle}. The lemma implies that $\Delta$ cuts off
a $1$-string trivial tangle and so $A$ is parallel to the frontier of the
boundary of a regular neighbourhood of a component of $t$, a contradiction.
So, $B^2\times 1$ contains at least two points of $\partial t$. By
Remark~\ref{rem:D-and-D'}, the same argument also implies that
$B^2\times 0$, as well, contains at least two points of $\partial t$. Suppose
that $B^2\times 1$ contains the three points $D'\cap \partial t
=D'\cap \partial(t_1\cup t_3)$. Then the cylinder $B^2\times I$ contains
$t_1\cup t_3$, because $A$ is disjoint from $t$. Since $B^2\times 0$ contains
at least two points of $\partial t$, it must contain a point of $\partial t_2$.
Hence we see that the cylinder $B^2\times I$ contains the whole
$t=t_1\cup t_2\cup t_3$. Since the arc $t_1$, which joins a point of
$B^2\times 0$ and a point of $B^2\times 1$ in the cylinder $B^2\times I$, is
unknotted, the cylinder is unknotted in $\threeball$ and hence the closure of
its complement is a solid torus. This implies that $A$ is parallel to the
frontier of a regular neighbourhood of $K_0$ in $\threeball$, a contradiction.
Hence both $B^2\times 0$ and $B^2\times 1$ contain exactly two points of
$\partial t$. Thus we see that $t_2\cup t_3$ is contained in the cylinder
$B^2\times I$ and that $t_1$ is contained in the complement of the cylinder.

\begin{Claim}
\label{c:disjoint-discs}
We may assume that $A$ is disjoint from the discs $\delta_2$ and $\delta_3$.
\end{Claim}

\begin{proof}
We may assume that the intersection of $A$ with $\delta_2\cup\delta_3$ is
transversal and the number of components of 
$A\cap(\delta_2\cup\delta_3)$ is 
minimised.
Suppose that $A\cap(\delta_2\cup\delta_3)$ contains loop components.
Pick a loop component, $C$, which is innermost in $\delta_2\cup\delta_3$,
and let $\delta_C$ be the subdisc of $\delta_2\cup\delta_3$ bounded by $C$.
Suppose $\delta_C$ contains a point of
$t\cap(\delta_2\cup\delta_3)=t_1\cap(\delta_2\cup\delta_3)$.
Then $C$ forms a meridian of $t_1$ and hence it is not null-homotopic in
$\threeball\setminus t$. So $C$ cannot be null-homotopic in the annulus
$A\subset \threeball\setminus t$ and so $C$ forms a core loop $A$. However, since $t_1$
is contained in the complement of the cylinder $B^2\times I$, the disc in the
cylinder bounded by $C$ is disjoint from $t_1$, and therefore $C$ is
null-homotopic in $\threeball\setminus t_1$, a contradiction. Hence $\delta_C$ is
disjoint from $t$. Since $A$ is incompressible, $C$ bounds a disc, $A_C$, in
$A$. The union $A_C\cup \delta_C$ is a a $2$-sphere in the irreducible manifold
$\threeball\setminus t$. Thus we can remove the intersection $C$ by using the $3$-ball
bounded by $A_C\cup \delta_C$. This contradicts the minimality of the
intersection $A\cap (\delta_2\cup\delta_3)$. Thus we have shown that
$A\cap (\delta_2\cup\delta_3)$ has no loop components, and so it is either
empty or consists of arcs whose boundaries are contained in
$\partial\threeball\cap (\delta_2\cup\delta_3)$.

Suppose that $A\cap(\delta_2\cup\delta_3)$ is nonempty,
and pick an arc
component, $C$, which is \lq\lq outermost'' in $\delta_2\cup\delta_3$. Let
$\delta_C$ be the ``outermost'' subdisc of $\delta_2\cup\delta_3$ bounded by
$C$ and the subarc of $\partial\threeball\cap (\delta_2\cup\delta_3)$ bounded
by $\partial C$. Since $\partial C$ lies  either in $D$ or $D'$, the arc $C$
must be inessential in $A$ and it cuts off a disc, $A_C$, from $A$. If
$\partial C\subset D$, then the union $\Delta:=A_C\cup \delta_C$ satisfies the
first two conditions of Lemma \ref{lem:no-boudary-compressing-disk}. Hence the
lemma implies that $\Delta$ is \lq\lq trivial'' and so we can remove the
intersection $C$. This contradicts the minimality of the intersection $A\cap 
(\delta_2\cup\delta_3)$. By Remark \ref{rem:D-and-D'}, the same argument works
when $\partial C\subset D'$. Thus we obtain the claim.
\end{proof}

By the above claim, $\partial B^2\times 0$ is isotopic in $\threeball\setminus t$ to
the boundary of a regular neighbourhood of $\delta_2\cap D$ in $D\setminus t_1$, and
$\partial B^2\times 1$ is isotopic in $\threeball\setminus t$ to the boundary of a
regular neighbourhood of $\delta_3\cap D'$ in $D'\setminus t_1$. Hence, the loop
$\partial B^2\times 0$ is null-homotopic in $\threeball\setminus (t_1\cup t_3)$,
and the loop $\partial B^2\times 1$ represents the conjugacy class of the
nontrivial element $((x_1x_3x_1^{-1})x_3^{-1})^{\pm 1}$ in
$\pi_1(\threeball\setminus (t_1\cup t_3))$
(see Figure \ref{fig:freegroup}).
This contradicts the fact that these loops
are freely homotopic in  $\threeball\setminus t$ and hence in
$\threeball\setminus (t_1\cup t_3)$. So, we have shown that Case 1 cannot happen.

\medskip
{\bf Case 2.} Both components of $\partial A$ are contained in $\interior D$,
and they bound mutually disjoint discs in $\interior D$. Then we have
$B^2\times \partial I\subset \interior D$. By the argument in the first step in
Case 1 by using the incompressibility of $A$ and
Lemma~\ref{lem:indivisible-tangle}, we can see that both $B^2\times 0$ and
$B^2\times 1$ contain at least two points of $\partial t$. Then $D$ must
contain at least four points of $\partial t$, a contradiction. By
Remark~\ref{rem:D-and-D'}, the same argument works when both components of
$\partial A$ are contained in $\interior D'$, and they bound mutually disjoint
discs in $\interior D'$.

\medskip
{\bf Case 3.} Both components of $\partial A$ are contained in $\interior D$,
and they are concentric in $\interior D$. Then we may assume that $B^2\times 0$
is contained in $\interior D$ and that $B^2\times 1$ contains $D'$ in its
interior. Reasoning as in Case 1, we see that $B^2\times 0$ contains at least
two endpoints of $t$, moreover $B^2\times 1$ must contain at last three of them
since it contains $D'$. As a consequence, the cylinder contains the entire
tangle $t$. 

\begin{Claim}
\label{c:bounding-solidtorus}
The cylinder $B^2\times I$ in unknotted in $\threeball$, and so
the closure of the complement of the cylinder $B^2\times I$
in $\threeball$ is a solid torus, $V$.
\end{Claim}

\begin{proof}
Assume first that the endpoints of $t_1$ in $D$ 
belong to $B^2\times 0$.
Then, since $t_1$ joins $\interior D$ with $\interior D'$, $t_1$ joins
$B^2\times 0$ with $B^2\times 1$ in the cylinder $B^2\times I$. 
So the cylinder
$B^2\times I$ is unknotted in $\threeball$ and hence the closure of its
complement in $\threeball$ is a solid torus $V$.

Suppose that the above assumption does not hold.
Then both endpoints of $t_1$ must belong to
$B^2\times 1$, while $B^2\times 0$ contains precisely the two endpoints of
$t_2$. 
Consider now the disc $\delta_2$ introduced before
Lemma~\ref{lem:no-boudary-compressing-disk}. Let $\eta$ be a small arc
contained in $\delta_2$ and joining the point $p$ of $t_1$ inside $\delta_2$ to
a point $q$ in the interior of $t_2$ (see Figure~\ref{fig:two-disks}). Let
$t_2'$ any of the two subarcs of $t_2$ going from one of its endpoints to $q$
and $t'_1$ the subarc of $t_1$ going from $p$ to the endpoint of $t_1$
contained in $D'$. We claim 
that the arc obtained by concatenating $t'_2$,
$\eta$, and $t'_1$ is trivial and contained in the cylinder $B^2\times I$.
Indeed, this arc cobounds a disc with an arc in $\partial \threeball$: such
disc is obtained by surgery along the two trivialising discs for $t_1$ and
$t_2$. To see that the arc is contained in the cylinder, it is enough to prove
that $\eta$ is contained inside the cylinder. This follows from the fact, that
can be proved as in Claim~\ref{c:disjoint-discs}, that $\delta_2$ does not meet
$A$. Using this trivial arc we see once more as in the previous case that the
cylinder is unknotted and its exterior is a solid torus $V$.
\end{proof}

By Claim \ref{c:bounding-solidtorus},
$(\threesphere\setminus \interior\threeball)\cup B^2\times I$ is a solid torus,
and $\partial B^2\times 0$ is its meridian.
Thus $\partial B^2\times 0$ is a longitude of the solid torus $V$.
So $A$ is parallel to the annulus $A':=V\cap \partial \threeball$ through $V$.
Since $A$ is essential in $\threeball\setminus t$, $V$ should contain a component of
$t$. This is, however, impossible for, as observed at the beginning, the
entire tangle is contained in the cylinder.

Now the proof of Lemma \ref{lem:no-essentail-annulus} is complete.
\end{proof}

Observe that the exterior of $t$ in $\threeball$ is a genus $3$ handlebody
and so $\threeball\setminus t$ is atoroidal
(i.e., does not contain an essential torus).
By using this fact and 
Lemmas \ref{lem:no-compressing-disk}-\ref{lem:no-essentail-annulus},
we can see that the link $\threesphere \setminus  L_3$ is atoroidal.
Since $L_3$ has $4$ components, this implies that 
$\threesphere \setminus  L_3$ cannot be a Seifert fibered space.
Hence $L_3$ is hyperbolic by Thurston's uniformization theorem for Haken manifolds.


\section{The link $L_6$ is hyperbolic}
\label{section:link6}


Consider a fundamental domain for the action of the dihedral group of order $6$
acting on $(\threesphere, L_6)$ and generated by the reflections in three
vertical planes, as illustrated in Figure~\ref{fig:link3}. The domain is shown
in Figure~\ref{fig:domain}.

\begin{figure}[h]
\begin{center}
 {
  \includegraphics[height=4cm]{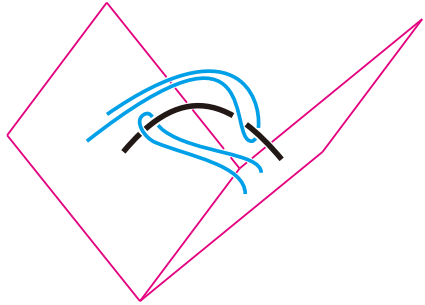}
 }
\end{center}
\caption{A local picture of the fundamental domain for the dihedral action on
$(\threesphere, L_6)$.}
\label{fig:domain}
\end{figure}

This domain can be seen as the intersection of two balls bounded by the
fixed-point sets of two reflections, that is two $2$-spheres. The two
$2$-spheres meet along a circle (a portion of the circle is the pink line of
intersection of the two planes in Figure~\ref{fig:domain}). 
It is not difficult to see that the
tangle in this fundamental domain, together with the circle of
intersection of the reflecting spheres on the boundary coincides with the
tangle and the equator $K_0$ on the right-hand side of Figure~\ref{fig:positive}.
Doubling this tangle we then obtain the sublink $\OO_3$ of the hyperbolic link
$L_3$, while the circle can be identified with $K_0$.

Since $L_3$ is hyperbolic, it is also $2\pi/3$-hyperbolic, for it is not the
figure-eight knot. Indeed, it follows from Thurston's orbifold theorem (see
\cite{BoP, CHK}) that a hyperbolic link that is not $2\pi/3$-hyperbolic
must be either Euclidean or spherical. Dunbar's list of geometric orbifolds
with underlying space the $3$-sphere \cite{D} shows that the only link with
this property is the figure-eight knot.

Consider now the hyperbolic orbifold $(\threesphere,L_3(2\pi/3))$: it is
topologically $\threesphere$ with singular set $L_3$ of cone angle $2\pi/3$
along every component. The reflection $\gamma_1$ of $L_3$ induces a hyperbolic
reflection of $(\threesphere,L_3(2\pi/3))$ along a totally geodesic surface.
This implies that the fundamental domain for the dihedral action admits a
hyperbolic structure with cone angle $2\pi/3$ along the components of the
tangle and cone angle $\pi/3$ along the circle cobounding the two
totally-geodesic three-punctured discs of the silvered boundary.

This cone hyperbolic structure lifts now to a cone hyperbolic structure of
$L_6$. As a consequence, the link $L_6$ is $2\pi/3$-hyperbolic and hence
hyperbolic.

\begin{remark}
Since the link $L_3$ is $2\pi/n$-hyperbolic for every $n\ge 3$, this
construction provides a whole family of hyperbolic links obtained by gluing
together $2n$ copies of the tangle in the fundamental domain, via reflections
in half of their boundaries. The resulting link $L_{2n}$ will admit a dihedral
symmetry of order $2n$. If $n$ is moreover odd, the link $L_{2n}$ will also
have the required $\G$-action to provide an admissible root.
\end{remark}

\paragraph{\bf Acknowledgements}
The main bulk of this research
was completed during a visit of 
the first named author
to Hiroshima
University in October 2017 funded by JSPS through FY2017 Invitational
Fellowship for Research in Japan (short term).
She
is also thankful to the Hiroshima University Math Department for
hospitality during her stay.
The authors started discussing the contents of this work during a
workshop held in Kyoto in 2011: they are grateful to the organisers, 
Teruaki Kitano and Takayuki Morifuji,
for inviting them.
They also thank Yasutaka Nakanishi for his helpful suggestions 
concerning the arguments in Sections~\ref{section:link2} and ~\ref{section:link3}.

\end{document}